\documentclass[10pt]{amsart}

\usepackage{geometry}
\usepackage{latexsym,exscale,enumerate}
\usepackage{amsmath,amsthm,amsfonts,amssymb,amscd, stmaryrd}
\usepackage{hyperref}

\usepackage[all]{xy}
\usepackage{tikz}

\usetikzlibrary{decorations.markings}
\usetikzlibrary{decorations.pathreplacing}

\newcommand{\tor}{{\Ss^1\times\Ss^1}}

\newcommand{\ai}{\overrightarrow{i}}
\newcommand{\aj}{\overrightarrow{j}}

\newcommand{\Ss}{\mathbb{S}}
\newcommand{\Aa}{\mathcal{A}}
\newcommand{\I}{\mathbb{I}}

\newcommand{\N}{\mathbb{N}}
\newcommand{\Z}{\mathbb{Z}}

\newcommand{\R}{\mathbb{R}}

\def\1{\mathbbm{1}}%

\newtheorem{theorem}{Theorem} 
\newtheorem{proposition}{Proposition} 
\newtheorem{definition}{Definition} 
\newtheorem{corollary}{Corollary} 
\newtheorem{example}{Example}

\theoremstyle{definition}

\title{Chebyshev polynomials and the Frohman-Gelca formula}

\author{Hoel Queffelec}
\address{Universit\'e Paris Diderot-Paris 7, Institut de Math\'ematiques de Jussieu - Paris Rive Gauche CNRS UMR 7586, B\^atiment Sophie Germain, Case 7012, 75205 Paris Cedex 13, France.}
\email{queffelec@math.jussieu.fr}

\author{Heather M. Russell}
\address{Department of Mathematics and Computer Science, Washington College, 300 Washington Avenue, Chestertown, MD 21620 U.S.A.}
\email{hrussell2@washcoll.edu}

\begin{document}

%

\maketitle

\begin{abstract}
Using Chebyshev polynomials, C. Frohman and R. Gelca introduce in \cite{FG} a basis of the Kauffman bracket skein module of the torus. This basis is especially useful because the Jones-Kauffman product can be described via a very simple \emph{Product-to-Sum} formula. Presented in this work is a diagrammatic proof of this formula, which emphasizes and demystifies the role played by Chebyshev polynomials.
\end{abstract}

\par

%
\section*{Introduction}
%

The Jones polynomial, originally defined in terms of Von Neumann algebras or representation theory of the quantum group $U_q(\mathfrak{sl}_2)$, also admits a diagrammatic description given by the Kauffman bracket~\cite{KauffBracket}. Starting from the diagram of a knot, the process is built on two easy local rules: one replaces each crossing by a formal sum of two different smoothings, and the other replaces each circle by the elementary polynomial $-A^2-A^{-2}$. At the end, one obtains a Laurent polynomial in the variable $A$ multiplied by the empty diagram. Adjusting this polynomial by a writhe factor, one gets the original Jones invariant.

This process was first defined for knot diagrams drawn on a disk, but one could also perform this process in situations where a knot diagram is drawn on an orientable surface. This would correspond to computations for knots embedded in the thickened surface. Because of the richer structure of the surface, knot diagrams can no longer always be reduced to the empty one. The skein module~\cite{Prz91} is thus defined as the natural place where the process ends.

Kauffman's local smoothing rules can be used to define the so-called Jones-Kauffman algebra structure on the skein module. This is induced by superposition of curves and smoothing of created crossings. The skein module of thickened surfaces and its related algebra structure have been extensively studied in works by Bullock, Frohman, Gelca, Kania-Bartoszynska, Przytycki~\cite{BuFK,BullPr,FG}. Although much information about this module can be obtained from homotopy of curves on the associated surface, it remains a difficult task to find a basis of the module in which the product can be easily described.

Frohman and Gelca \cite{FG} addressed this problem in the torus case introducing a new basis of the skein module where the result of the product of two basis elements is a sum of only two other basis elements. This elegant and somewhat surprising formula relies on an extensive but somewhat enigmatic use of Chebyshev's polynomial of the first kind. Sallenave~\cite{Sal1} independently and without using the Chebyshev polynomials obtained similar results, which later led to the idea of an \emph{oriented} skein module as it appears in \cite{MR1834496}. In the present work, we study and use these diagrammatic tools to give a new proof of Frohman and Gelca's formula with a focus on clarifying the role of the Chebyshev polynomials as a change of basis. The argument presented here clearly splits the proof of Frohman-Gelca's formula into a technical result which does not involve Chebyshev polynomials and a change of basis argument where they are the key ingredient. Inspirational ideas can also be found in works of Przytycki~\cite{Prz98}.

It is very interesting to note that the module we introduce actually comes from the categorification proposed by Asaeda, Przytycki and Sikora~\cite{APS}: the categorification of the skein module turns out to highlight some of the hidden structure of the original object! Our interest in developing a better understanding of the product in the Jones-Kauffman algebra is largely motivated by the hope that it could be lifted to the categorified level and, for example, play a role in a surgery process for 3-manifold extensions of Khovanov homology.

Several recent works have studied the interplay between skein algebras and cluster algebras \cite{FoPy,ThurstonSkein} often taking the quantum variable to be $1$. The surprising resemblance between the Frohman-Gelca formula and a mutation formula is also a motivation for the study of related modules. The \emph{Product-to-Sum} formula itself has also been recently receiving a growing interest in particular at roots of unity \cite{BW1,BW2}. Indeed, Bonahon-Wong's {\it miraculous cancellation} extensively uses the Frohman-Gelca formula in the punctured torus, which only holds with an adequate choice of $A$ as a root of the unity. The oriented skein module presented here can also be extended to any orientable punctured surface. Its symmetric part then appears as a quotient of the entire skein module, but it allows one to keep $A$ generic and to have generalizations of the Frohman-Gelca formula~\cite{Queff_PhD}. It remains to be seen if nice presentations via generators and relations of the oriented skein module exist beyond the torus case.

\vspace{.5cm}

\noindent \textbf{Acknowledgements:} We wish to thank Charlie Frohman for his inspiring suggestions, and Francis Bonahon and Helen Wong for interesting discussions. H. Q. also wants to thank his advisors Christian Blanchet and Catharina Stroppel for their constant support and numerous comments on this work, as well as Louis Funar for all his comments. H. R. was supported by the John Templeton Foundation.

%
\section{Jones-Kauffman algebra}
%

%
\subsection{Usual skein module: the plane case}
%

In 1984, Jones~\cite{Jones} introduced a new polynomial invariant for oriented knots with origins coming from Von Neumann algebras. Kauffman~\cite{KauffBracket} developed a straightforward, combinatorial process for obtaining this polynomial in which one first computes what we now refer to as the \emph{Kauffman bracket} corresponding to a knot diagram and then multiplies by a certain writhe factor. To calculate the Kauffman bracket, we proceed by recursively resolving crossings and removing circles  according to rules \ref{KauffBrSmoothing} and \ref{KauffBrCircle}  to produce a Laurent polynomial in the variable $A$. The Kauffman bracket is independent of orientation of a diagram.

\begin{equation}
\xy
 (0,0)*{
\begin{tikzpicture} [decoration={markings,
                         mark=at position 0.6 with {\arrow{>}};    }]
\draw[dashed] (.75,0) arc(0:360:.75);
\draw[very thick] (.5,-.5) -- (-.5,.5);
\draw[white, double, double distance=1.25pt, double=black, ultra thick] (-.5,-.5) -- (.5,.5); 
\end{tikzpicture}};
\endxy
\;\; =\;\; A\;\;
\xy
 (0,0)*{
\begin{tikzpicture} [decoration={markings,
                         mark=at position 0.6 with {\arrow{>}};    }]
\draw[dashed] (.75,0) arc(0:360:.75);
\draw[very thick] (-.5,-.5) .. controls (0,0) .. (-.5,.5);
\draw[very thick] (.5,-.5) .. controls (0,0) .. (.5,.5); 
\end{tikzpicture}};
\endxy
\;\; + \;\; A^{-1}\;\;
\xy
 (0,0)*{
\begin{tikzpicture} [decoration={markings,
                         mark=at position 0.6 with {\arrow{>}};    }]
\draw[dashed] (.75,0) arc(0:360:.75);
\draw[very thick] (-.5,-.5) .. controls (0,0) .. (.5,-.5);
\draw[very thick] (-.5,.5) .. controls (0,0) .. (.5,.5); 
\end{tikzpicture}};
\endxy \label{KauffBrSmoothing}
\end{equation}

\begin{equation}
\xy
 (0,0)*{
\begin{tikzpicture} [decoration={markings,
                         mark=at position 0.6 with {\arrow{>}};    }]
\draw[dashed] (.75,0) arc(0:360:.75);
\draw[very thick] (.5,0) arc(0:360:.5);
\end{tikzpicture}};
\endxy
\;\; =\;(-A^2-A^{-2}) \;\;
\xy
 (0,0)*{
\begin{tikzpicture} [decoration={markings,
                         mark=at position 0.6 with {\arrow{>}};    }]
\draw[dashed] (.75,0) arc(0:360:.75);
\end{tikzpicture}};
\endxy \label{KauffBrCircle}
\end{equation}

Indeed using rule \ref{KauffBrSmoothing} to smooth all crossings, one obtains a formal sum of unions of closed curves in the plane with coefficients in the ring $R=\Z[A,A^{-1}]$. Since closed curves in the plane are nothing but circles, applying rule \ref{KauffBrCircle} (where disjoint union is replaced by product) removes all curves from the sum yielding a single polynomial multiplied by {\it the empty curve}.

Starting from a knot diagram in the plane, one defines the Kauffman bracket to be the Laurent polynomial that appears as the coefficient of {\it the empty curve} once rules \ref{KauffBrSmoothing} and \ref{KauffBrCircle} have been applied exhaustively. This polynomial is invariant under Reidemeister moves II and III, and thus it is a so-called framed knot invariant. As mentioned above, for an oriented diagram the Kauffman bracket scaled by a certain writhe factor is a combinatorial version of the Jones polynomial. As an example, we compute the Kauffman Bracket of the Hopf link below.

\begin{example} \label{ex:KauffBracket}
We start by smoothing all crossings using rule \ref{KauffBrSmoothing}

\begin{eqnarray}
\xy
(0,0)*{
\begin{tikzpicture}[scale=1]
\draw [very thick] (0,0) arc (-180:0:.5);
\draw [very thick] (.8,0) arc (180:0:.5);
\draw [thick, white, double, double=black, double distance=1.25pt] (0,0) arc (180:0:.5);
\draw [thick, white, double, double=black, double distance=1.25pt] (.8,0) arc (-180:0:.5);
\end{tikzpicture}
};
\endxy
\quad &\mapsto& \quad
A\;
\xy
(0,0)*{
\begin{tikzpicture}[scale=1]
\draw [very thick] (0,0).. controls (0,-.8) and (.8,-.8) .. (.8,-.5) .. controls (.8,-.3) and (1,-.3) .. (1,-.5) .. controls (1,-.8) and (1.8,-.8) .. (1.8,0);
\draw [very thick] (.8,0) arc (180:0:.5);
\draw [very thick] (1,0) .. controls (1,-.3) and (.8,-.3) .. (.8,0);
\draw [thick, white, double, double=black, double distance=1.25pt] (0,0) arc (180:0:.5);
\end{tikzpicture}
};
\endxy
\; + A^{-1}\;
\xy
(0,0)*{
\begin{tikzpicture}[scale=1.2]
\draw [very thick] (0,0).. controls (0,-.8) and (.85,-.6) .. (.85,-.3) .. controls (.85,-.2) and (.8,-.1) .. (.8,0);
\draw [very thick] (1.8,0) .. controls (1.8,-.8) and (.95,-.6) .. (.95,-.3) .. controls (.95,-.2) and (1,-.1) .. (1,0);
\draw [very thick] (.8,0) arc (180:0:.5);
\draw [thick, white, double, double=black, double distance=1.25pt] (0,0) arc (180:0:.5);
\end{tikzpicture}
};
\endxy \nonumber \\
&\mapsto&  \quad
A^2\;
\xy
(0,0)*{
\begin{tikzpicture}[scale=1]
\draw [very thick] (0,0).. controls (0,-.8) and (.8,-.8) .. (.8,-.5) .. controls (.8,-.3) and (1,-.3) .. (1,-.5) .. controls (1,-.8) and (1.8,-.8) .. (1.8,0);
\draw [very thick] (0,0).. controls (0,.8) and (.8,.8) .. (.8,.5) .. controls (.8,.3) and (1,.3) .. (1,.5) .. controls (1,.8) and (1.8,.8) .. (1.8,0);
\draw [very thick] (1,0) .. controls (1,-.3) and (.8,-.3) .. (.8,0);
\draw [very thick] (1,0) .. controls (1,.3) and (.8,.3) .. (.8,0);
\end{tikzpicture}
};
\endxy
\;+\;
\xy
(0,0)*{
\begin{tikzpicture}[scale=1]
\draw [very thick] (0,0).. controls (0,.8) and (.85,.6) .. (.85,.3) .. controls (.85,.2) and (.8,.1) .. (.8,0);
\draw [very thick] (1.8,0) .. controls (1.8,.8) and (.95,.6) .. (.95,.3) .. controls (.95,.2) and (1,.1) .. (1,0);
\draw [very thick] (0,0).. controls (0,-.8) and (.8,-.8) .. (.8,-.5) .. controls (.8,-.3) and (1,-.3) .. (1,-.5) .. controls (1,-.8) and (1.8,-.8) .. (1.8,0);
\draw [very thick] (1,0) .. controls (1,-.3) and (.8,-.3) .. (.8,0);
\end{tikzpicture}
};
\endxy
\;+\;
\xy
(0,0)*{
\begin{tikzpicture}[scale=1]
\draw [very thick] (0,0).. controls (0,-.8) and (.85,-.6) .. (.85,-.3) .. controls (.85,-.2) and (.8,-.1) .. (.8,0);
\draw [very thick] (1.8,0) .. controls (1.8,-.8) and (.95,-.6) .. (.95,-.3) .. controls (.95,-.2) and (1,-.1) .. (1,0);
\draw [very thick] (0,0).. controls (0,.8) and (.8,.8) .. (.8,.5) .. controls (.8,.3) and (1,.3) .. (1,.5) .. controls (1,.8) and (1.8,.8) .. (1.8,0);
\draw [very thick] (1,0) .. controls (1,.3) and (.8,.3) .. (.8,0);
\end{tikzpicture}
};
\endxy
\;+\;
A^{-2}\;
\xy
(0,0)*{
\begin{tikzpicture}[scale=1]
\draw [very thick] (0,0).. controls (0,-.8) and (.85,-.6) .. (.85,-.3) .. controls (.85,-.2) and (.8,-.1) .. (.8,0);
\draw [very thick] (1.8,0) .. controls (1.8,-.8) and (.95,-.6) .. (.95,-.3) .. controls (.95,-.2) and (1,-.1) .. (1,0);
\draw [very thick] (0,0).. controls (0,.8) and (.85,.6) .. (.85,.3) .. controls (.85,.2) and (.8,.1) .. (.8,0);
\draw [very thick] (1.8,0) .. controls (1.8,.8) and (.95,.6) .. (.95,.3) .. controls (.95,.2) and (1,.1) .. (1,0);
\end{tikzpicture}
};
\endxy \nonumber
\end{eqnarray}

Next, we evaluate the circles according to rule \ref{KauffBrCircle}.

\begin{eqnarray}
\xy
(0,0)*{
\begin{tikzpicture}[scale=1]
\draw [very thick] (0,0) arc (-180:0:.5);
\draw [very thick] (.8,0) arc (180:0:.5);
\draw [thick, white, double, double=black, double distance=1.25pt] (0,0) arc (180:0:.5);
\draw [thick, white, double, double=black, double distance=1.25pt] (.8,0) arc (-180:0:.5);
\end{tikzpicture}
};
\endxy
&\mapsto&
A^{2}(-A^2-A^{-2})^2+(-A^{2}-A^{-2})+(-A^{2}-A^{-2})+A^{-2}(-A^2-A^{-2})^2\nonumber \\
& &=A^6+A^2+A^{-2}+A^{-6}. \nonumber
\end{eqnarray}

\end{example}

%
\subsection{Passing to surfaces}
%

The previous computational rules are entirely local, so it should be possible to apply them to knot diagrams drawn on more general surfaces. A knot diagram in the plane is a projection of a knot in $\R^3$ or $\R^2\times \I$ (where $\I=[0,1]$). A knot diagram on any surface $S$, then, should be a projection of a knot in $S\times \I$ which we sometimes refer to as {\it the thickening of $S$}.

\begin{definition}
Let $S$ be a surface. A knot in the thickening of $S$ is the embedding of a circle $\Ss^1$ in $S\times \I$ considered up to isotopy of the whole $3$-manifold $S\times \I$. Similarly, a link in $S\times \I$ is the embedding of a disjoint union $\sqcup_k\Ss^1$ in $S\times \I$. A framed knot (or link) is a knot (or link) with an orthogonal non-vanishing vector field, which can be thought to as the embedding of a band $\Ss^1\times \I$ (or copies of such bands).
\end{definition}

As in the planar case, all diagrams of a given knot on $S$ are related by Reidemeister moves and isotopies of $S$. In the framed case, just as usual, a diagram can correspond to a framed knot by adopting the blackboard framing rule, in which case two diagrams are equivalent if related through Reidemeister moves II and III, isotopies, and the framed version of the first Reidemeister move.

\begin{eqnarray}
\xy
(0,0)*{
\begin{tikzpicture}[scale=0.3]
\draw[dashed] (5.7,-3.5) arc(0:360:4);
\draw[style=very thick, -] (0,0) -- (0,-2);
\draw [style=very thick, -](0,-2) arc (180:360: 1cm);
\draw [style=very thick, -](2,-2) arc (0:160: .9cm);
\draw[style=very thick, -] (0.1,-2.7) -- (0.1,-4.3);
\draw [style=very thick, -](2,-5) arc (0:-150: 0.9cm);
\draw [style=very thick, -](2,-5) arc (0:180: 1cm);
\draw[style=very thick, -] (0,-5) -- (0,-7);
\end{tikzpicture}
};
\endxy
\;\;\underleftrightarrow{\;R'_I\;}\;\;
\xy
(0,0)*{
\begin{tikzpicture}[scale=0.3]
\draw[dashed] (5.7,-3.5) arc(0:360:4);
\draw[style=very thick, -] (0,0) .. controls (2,-3) and (2,-4) .. (0,-7);
\end{tikzpicture}
};
\endxy
\label{Framed_R1}
\end{eqnarray}

The Kauffman bracket also extends to our new setting as a framed invariant. Below we calculate an example on the torus. Note that the small circle represents a homologically trivial circle in the thickened torus.

\begin{example}
We start from a diagram similar to the one from Example \ref{ex:KauffBracket}, but now considered on a torus. We once again begin by smoothing all crossings according to rule \ref{KauffBrSmoothing}.

\begin{eqnarray}
\xy
(0,0)*{
  \begin{tikzpicture} [scale=.4,fill opacity=0.2,  decoration={markings, 
                        mark=at position 0.6 with {\arrow{>}};    }]
\draw (-4,0) .. controls (-4,1.2) and (-2,2) .. (0,2) .. controls (2,2) and (4,1.2) .. (4,0) .. controls (4,-1.2) and (2,-2) .. (0,-2) .. controls (-2,-2) and (-4,-1.2) .. (-4,0);
     \draw (-2,0.2) arc (-120:-60:4);
     \draw (1.36,-0.08) arc (70:110:4);
\draw [very thick] (0,-.6) arc (90:-90:.5);
\draw [very thick] (0,-1.1) .. controls (-3.5,-1.1) and (-3.5,1.1) .. (0,1.1);
\draw [very thick, white, double, double=black, double distance=1.25pt] (0,-1.1) .. controls (3.5,-1.1) and (3.5,1.1) .. (0,1.1);
\draw [very thick, white, double, double=black, double distance=1.25pt] (0,-.6) arc (90:270:.5);
    \end{tikzpicture}};
\endxy
\quad &\mapsto& \quad
A\;
\xy
(0,0)*{
  \begin{tikzpicture} [scale=.4,fill opacity=0.2,  decoration={markings, 
                        mark=at position 0.6 with {\arrow{>}};    }]
\draw (-4,0) .. controls (-4,1.2) and (-2,2) .. (0,2) .. controls (2,2) and (4,1.2) .. (4,0) .. controls (4,-1.2) and (2,-2) .. (0,-2) .. controls (-2,-2) and (-4,-1.2) .. (-4,0);
     \draw (-2,0.2) arc (-120:-60:4);
     \draw (1.36,-0.08) arc (70:110:4);
\draw [very thick] (0,-.6) arc (90:-90:.5);
\draw [very thick] (0,-.6) .. controls (-.5,-.6) and (-.5,-1.1) .. (0,-1.1);
\draw [very thick] (0,-1.6) .. controls (-.4,-1.6) and (-.3,-1.1) ..(-.6,-1.1) .. controls (-3.5,-1.1) and (-3.5,1.1) .. (0,1.1);
\draw [very thick, white, double, double=black, double distance=1.25pt] (0,-1.1) .. controls (3.5,-1.1) and (3.5,1.1) .. (0,1.1);
    \end{tikzpicture}};
\endxy
\;+
A^{-1}\;
\xy
(0,0)*{
  \begin{tikzpicture} [scale=.4,fill opacity=0.2,  decoration={markings, 
                        mark=at position 0.6 with {\arrow{>}};    }]
\draw (-4,0) .. controls (-4,1.2) and (-2,2) .. (0,2) .. controls (2,2) and (4,1.2) .. (4,0) .. controls (4,-1.2) and (2,-2) .. (0,-2) .. controls (-2,-2) and (-4,-1.2) .. (-4,0);
     \draw (-2,0.2) arc (-120:-60:4);
     \draw (1.36,-0.08) arc (70:110:4);
\draw [very thick] (0,-.6) arc (90:-90:.5);
\draw [very thick] (0,-1.6) .. controls (-.5,-1.6) and (-.5,-1.1) .. (0,-1.1);
\draw [very thick] (0,-.6) .. controls (-.4,-.6) and (-.3,-1.1) ..(-.6,-1.1) .. controls (-3.5,-1.1) and (-3.5,1.1) .. (0,1.1);
\draw [very thick, white, double, double=black, double distance=1.25pt] (0,-1.1) .. controls (3.5,-1.1) and (3.5,1.1) .. (0,1.1);
    \end{tikzpicture}};
\endxy\nonumber \\
&\mapsto&
A^2\;
\xy
(0,0)*{
  \begin{tikzpicture} [scale=.4,fill opacity=0.2,  decoration={markings, 
                        mark=at position 0.6 with {\arrow{>}};    }]
\draw (-4,0) .. controls (-4,1.2) and (-2,2) .. (0,2) .. controls (2,2) and (4,1.2) .. (4,0) .. controls (4,-1.2) and (2,-2) .. (0,-2) .. controls (-2,-2) and (-4,-1.2) .. (-4,0);
     \draw (-2,0.2) arc (-120:-60:4);
     \draw (1.36,-0.08) arc (70:110:4);
\draw [very thick] (0,-.6) .. controls (-.5,-.6) and (-.5,-1.1) .. (0,-1.1);
\draw [very thick] (0,-.6) .. controls (.5,-.6) and (.5,-1.1) .. (0,-1.1);
\draw [very thick] (0,-1.6) .. controls (-.4,-1.6) and (-.3,-1.1) ..(-.6,-1.1) .. controls (-3.5,-1.1) and (-3.5,1.1) .. (0,1.1);
\draw [very thick] (0,-1.6) .. controls (.4,-1.6) and (.3,-1.1) ..(.6,-1.1) .. controls (3.5,-1.1) and (3.5,1.1) .. (0,1.1);
    \end{tikzpicture}};
\endxy
\;+\;
\xy
(0,0)*{
  \begin{tikzpicture} [scale=.4,fill opacity=0.2,  decoration={markings, 
                        mark=at position 0.6 with {\arrow{>}};    }]
\draw (-4,0) .. controls (-4,1.2) and (-2,2) .. (0,2) .. controls (2,2) and (4,1.2) .. (4,0) .. controls (4,-1.2) and (2,-2) .. (0,-2) .. controls (-2,-2) and (-4,-1.2) .. (-4,0);
     \draw (-2,0.2) arc (-120:-60:4);
     \draw (1.36,-0.08) arc (70:110:4);
\draw [very thick] (0,-.6) .. controls (-.5,-.6) and (-.5,-1.1) .. (0,-1.1);
\draw [very thick] (0,-1.1) .. controls (.5,-1.1) and (.5,-1.6) .. (0,-1.6);
\draw [very thick] (0,-1.6) .. controls (-.4,-1.6) and (-.3,-1.1) ..(-.6,-1.1) .. controls (-3.5,-1.1) and (-3.5,1.1) .. (0,1.1);
\draw [very thick] (0,-.6) .. controls (.4,-.6) and (.3,-1.1) ..(.6,-1.1) .. controls (3.5,-1.1) and (3.5,1.1) .. (0,1.1);
    \end{tikzpicture}};
\endxy \nonumber \\
&+&
\xy
(0,0)*{
  \begin{tikzpicture} [scale=.4,fill opacity=0.2,  decoration={markings, 
                        mark=at position 0.6 with {\arrow{>}};    }]
\draw (-4,0) .. controls (-4,1.2) and (-2,2) .. (0,2) .. controls (2,2) and (4,1.2) .. (4,0) .. controls (4,-1.2) and (2,-2) .. (0,-2) .. controls (-2,-2) and (-4,-1.2) .. (-4,0);
     \draw (-2,0.2) arc (-120:-60:4);
     \draw (1.36,-0.08) arc (70:110:4);
\draw [very thick] (0,-.6) .. controls (.5,-.6) and (.5,-1.1) .. (0,-1.1);
\draw [very thick] (0,-1.1) .. controls (-.5,-1.1) and (-.5,-1.6) .. (0,-1.6);
\draw [very thick] (0,-1.6) .. controls (.4,-1.6) and (.3,-1.1) ..(.6,-1.1) .. controls (3.5,-1.1) and (3.5,1.1) .. (0,1.1);
\draw [very thick] (0,-.6) .. controls (-.4,-.6) and (-.3,-1.1) ..(-.6,-1.1) .. controls (-3.5,-1.1) and (-3.5,1.1) .. (0,1.1);
    \end{tikzpicture}};
\endxy
\;+A^{-2}\;
\xy
(0,0)*{
  \begin{tikzpicture} [scale=.4,fill opacity=0.2,  decoration={markings, 
                        mark=at position 0.6 with {\arrow{>}};    }]
\draw (-4,0) .. controls (-4,1.2) and (-2,2) .. (0,2) .. controls (2,2) and (4,1.2) .. (4,0) .. controls (4,-1.2) and (2,-2) .. (0,-2) .. controls (-2,-2) and (-4,-1.2) .. (-4,0);
     \draw (-2,0.2) arc (-120:-60:4);
     \draw (1.36,-0.08) arc (70:110:4);
\draw [very thick] (0,-1.6) .. controls (.5,-1.6) and (.5,-1.1) .. (0,-1.1);
\draw [very thick] (0,-1.6) .. controls (-.5,-1.6) and (-.5,-1.1) .. (0,-1.1);
\draw [very thick] (0,-.6) .. controls (.4,-.6) and (.3,-1.1) ..(.6,-1.1) .. controls (3.5,-1.1) and (3.5,1.1) .. (0,1.1);
\draw [very thick] (0,-.6) .. controls (-.4,-.6) and (-.3,-1.1) ..(-.6,-1.1) .. controls (-3.5,-1.1) and (-3.5,1.1) .. (0,1.1);
    \end{tikzpicture}};
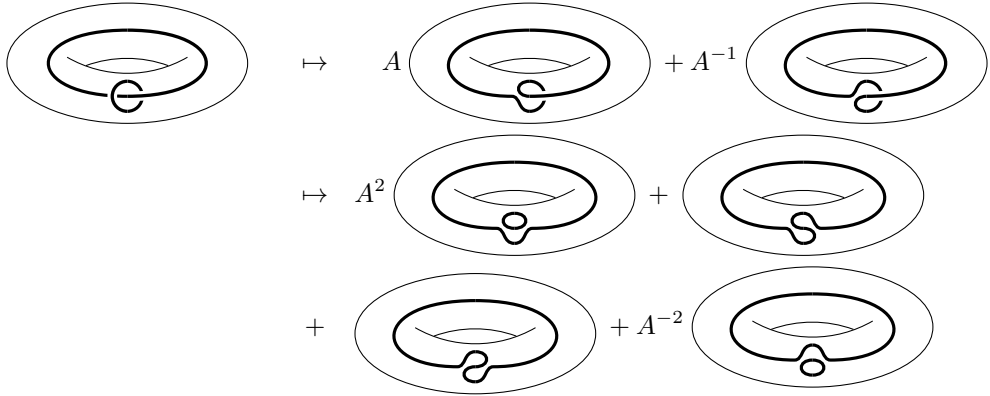
\endxy \nonumber
\end{eqnarray}

Then we remove all possible null-homotopic circles using rule \ref{KauffBrCircle}.

\begin{eqnarray}
\xy
(0,0)*{
  \begin{tikzpicture} [scale=.4,fill opacity=0.2,  decoration={markings, 
                        mark=at position 0.6 with {\arrow{>}};    }]
\draw (-4,0) .. controls (-4,1.2) and (-2,2) .. (0,2) .. controls (2,2) and (4,1.2) .. (4,0) .. controls (4,-1.2) and (2,-2) .. (0,-2) .. controls (-2,-2) and (-4,-1.2) .. (-4,0);
     \draw (-2,0.2) arc (-120:-60:4);
     \draw (1.36,-0.08) arc (70:110:4);
\draw [very thick] (0,-.6) arc (90:-90:.5);
\draw [very thick] (0,-1.1) .. controls (-3.5,-1.1) and (-3.5,1.1) .. (0,1.1);
\draw [very thick, white, double, double=black, double distance=1.25pt] (0,-1.1) .. controls (3.5,-1.1) and (3.5,1.1) .. (0,1.1);
\draw [very thick, white, double, double=black, double distance=1.25pt] (0,-.6) arc (90:270:.5);
    \end{tikzpicture}};
\endxy
&\mapsto&
\text{{\small $(A^{2}(-A^{2}-A^{-2})+1+1+A^{-2}((-A^{2}-A^{-2}))$}}\;
\xy
(0,0)*{
  \begin{tikzpicture} [scale=.4,fill opacity=0.2,  decoration={markings, 
                        mark=at position 0.6 with {\arrow{>}};    }]
\draw (-4,0) .. controls (-4,1.2) and (-2,2) .. (0,2) .. controls (2,2) and (4,1.2) .. (4,0) .. controls (4,-1.2) and (2,-2) .. (0,-2) .. controls (-2,-2) and (-4,-1.2) .. (-4,0);
     \draw (-2,0.2) arc (-120:-60:4);
     \draw (1.36,-0.08) arc (70:110:4);
\draw [very thick] (0,-1.1) .. controls (3.5,-1.1) and (3.5,1.1) .. (0,1.1);
\draw [very thick] (0,-1.1) .. controls (-3.5,-1.1) and (-3.5,1.1) .. (0,1.1);
    \end{tikzpicture}};
\endxy \nonumber \\
&&=
\text{{\small $(-A^{4}-A^{-4})$}}\;
\xy
(0,0)*{
  \begin{tikzpicture} [scale=.4,fill opacity=0.2,  decoration={markings, 
                        mark=at position 0.6 with {\arrow{>}};    }]
\draw (-4,0) .. controls (-4,1.2) and (-2,2) .. (0,2) .. controls (2,2) and (4,1.2) .. (4,0) .. controls (4,-1.2) and (2,-2) .. (0,-2) .. controls (-2,-2) and (-4,-1.2) .. (-4,0);
     \draw (-2,0.2) arc (-120:-60:4);
     \draw (1.36,-0.08) arc (70:110:4);
\draw [very thick] (0,-1.1) .. controls (3.5,-1.1) and (3.5,1.1) .. (0,1.1);
\draw [very thick] (0,-1.1) .. controls (-3.5,-1.1) and (-3.5,1.1) .. (0,1.1);
    \end{tikzpicture}};
\endxy \nonumber
\end{eqnarray}
\end{example}

As this example shows, the Kauffman bracket for general surfaces is more complicated than in the planar case. Indeed, the Kauffman bracket is no longer necessarily an element of $R$ (multiplied by the empty curve) but instead an $R$-linear combination of homotopically nontrivial curves on $S$.  This observation is the key idea leading to the definition of the \emph{Kauffman bracket skein module} (see for instance \cite{Prz91} and the survey~\cite{Prz_skein}).

\begin{definition}
Let $S$ be a surface. The \emph{Kauffman bracket skein module} of $S$, denoted $Sk(S)$ is the $R$-module generated by unoriented, closed, embedded curves on $S$ up to isotopy subject to the local relation in rule \ref{KauffBrCircle}. (We emphasize that only homotopically trivial circles (i.e.  circles bounding a disk) can be deleted using rule \ref{KauffBrCircle}.)
\end{definition}

This module is a natural context for studying the Kauffman bracket of knots in thickened surfaces. The Kauffman bracket of a diagram on $S$ is a skein element in $Sk(S)$, and moreover, each framed isotopy class of links in $S\times\I$ has a unique skein element associated to it. In other words, the Kauffman bracket is an invariant of framed links.

%
\subsection{The algebra structure}
%

The smoothing process that defines the Kauffman bracket also endows $Sk(S)$ with an algebra structure: the product of curves $\alpha$ and $\beta$ is the skein element in $Sk(S)$ corresponding to any non-singular superposition of $\alpha$ over $\beta$. The following definition makes this precise.

\begin{definition}
Let $\alpha, \beta\in Sk(S)$ be generators corresponding to embeddings $\alpha: \sqcup_k \Ss^1 \rightarrow S$ and $\beta: \sqcup_l \Ss^1 \rightarrow S$. Define $\alpha \ast \beta$ to be the element of $Sk(S)$ corresponding to the framed curve in $S\times I$ described by the following mapping
\[
\sqcup_{k+l}\Ss^1\mapsto S\times \I,
\]
where $x\mapsto (\alpha(x), 1)$ if $x$ is in one of the first $k$ copies of $\Ss^1$ and $x\mapsto (\beta(x), 0)$ otherwise. This product extends linearly.

\end{definition}

As an example, we calculate the product of two curves on the torus.

\begin{example}

\begin{equation}
\xy (0,0)*{
  \begin{tikzpicture} [scale=.5,fill opacity=0.2,  decoration={markings, 
                        mark=at position 0.6 with {\arrow{>}};    }]
     \draw [domain=-180:180, smooth, tension=1] plot ({2*cos(\x)},{sin(\x )});
     \draw (-1,0.1) arc (-120:-60:2) -- (0.68,-0.04) arc (70:110:2) -- cycle;
     \draw [very thick, domain=-180:180, smooth, tension=1] plot ({1.5*cos(\x)},{.7*sin(\x )});
    \end{tikzpicture}};
\endxy
\; \ast \;
\xy (0,0)*{
  \begin{tikzpicture} [scale=.5,fill opacity=0.2,  decoration={markings, 
                        mark=at position 0.6 with {\arrow{>}};    }]
     \draw [domain=-180:180, smooth, tension=1] plot ({2*cos(\x)},{sin(\x )});
     \draw (-1,0.1) arc (-120:-60:2) -- (0.68,-0.04) arc (70:110:2) -- cycle;
     \draw [very thick] (0,-1) .. controls (.4,-1) and (.4,-.16) .. (0,-.16);
     \draw [dotted] (0,-1) .. controls (-.3,-1) and (-.3,-.16) .. (0,-.16);
    \end{tikzpicture}};
\endxy
\; =\;
\xy (0,0)*{
  \begin{tikzpicture} [scale=.5,fill opacity=0.2,  decoration={markings, 
                        mark=at position 0.6 with {\arrow{>}};    }]
     \draw [domain=-180:180, smooth, tension=1] plot ({2*cos(\x)},{sin(\x )});
     \draw (-1,0.1) arc (-120:-60:2) -- (0.68,-0.04) arc (70:110:2) -- cycle;
     \draw [very thick] (0,-1) .. controls (.4,-1) and (.4,-.16) .. (0,-.16);
     \draw [dotted] (0,-1) .. controls (-.3,-1) and (-.3,-.16) .. (0,-.16);
      \draw [ultra thick, double, color= white, double=black,double distance=1.25pt, domain=-180:180, smooth, tension=1] plot ({1.5*cos(\x)},{.7*sin(\x )});
    \end{tikzpicture}};
\endxy
\;= A\;
\xy (0,0)*{
  \begin{tikzpicture} [scale=.5,fill opacity=0.2,  decoration={markings, 
                        mark=at position 0.6 with {\arrow{>}};    }]
     \draw [domain=-180:180, smooth, tension=1] plot ({2*cos(\x)},{sin(\x )});
     \draw (-1,0.1) arc (-120:-60:2) -- (0.68,-0.04) arc (70:110:2) -- cycle;
      \draw [very thick, domain=-60:240, smooth, tension=1] plot ({1.5*cos(\x)},{.7*sin(\x )});
     \draw [very thick] (0.75,-0.6) .. controls (.5,-.7) and (.4,-1) .. (.2,-1);
     \draw [very thick] (-0.75,-0.6) .. controls (-0.5,-.7) and (-.2,-.1) .. (0,-.16);
    \draw [dotted] (.2,-1) .. controls (0,-1) and (.2,-.16) .. (0,-.16);
    \end{tikzpicture}};
\endxy
\; + A^{-1} \;
\xy (0,0)*{
  \begin{tikzpicture} [scale=.5,fill opacity=0.2,  decoration={markings, 
                        mark=at position 0.6 with {\arrow{>}};    }]
     \draw [domain=-180:180, smooth, tension=1] plot ({2*cos(\x)},{sin(\x )});
     \draw (-1,0.1) arc (-120:-60:2) -- (0.68,-0.04) arc (70:110:2) -- cycle;
      \draw [very thick, domain=-60:240, smooth, tension=1] plot ({1.5*cos(\x)},{.7*sin(\x )});
     \draw [very thick] (0.75,-0.6) .. controls (0.5,-.7) and (.2,-.1) .. (0,-.16);
     \draw [very thick] (-0.75,-0.6) .. controls (-0.5,-.7) and (-.4,-1) .. (-.2,-1);
    \draw [dotted] (-.2,-1) .. controls (0,-1) and (-.2,-.16) .. (0,-.16);
    \end{tikzpicture}};
\endxy
\label{eq:Product11}
\end{equation}

\end{example}

Consider $Sk(\R^2)$ endowed with this product. When performing a multiplication, the top element can always be pushed away so that it does not lie over the bottom one. This means that the product in this case is just disjoint union of curves. Note also that if one multiplies two copies of the same curve (on an orientable surface), we can realize it as a disjoint union.

%
\subsection{Conventions for toric curves}
%

For the remainder of our discussion, we are interested in the case where $S$ is the torus.  We fix an oriented longitude $\lambda$ and meridian $\mu$ as well as an orientation on $\tor$ so that, on a neighborhood of the intersection between $\lambda$ and $\mu$, their germs in this order form a negative basis of the tangent space. This is summarized in Figure \ref{torusconventions}. 
\begin{figure}[h]
\[
\xy
(0,0)*{
\begin{tikzpicture} [scale=.5,  decoration={markings, 
                        mark=at position 0.6 with {\arrow{>}};    }]
\draw (-4,0) .. controls (-4,1.2) and (-2,2) .. (0,2) .. controls (2,2) and (4,1.2) .. (4,0) .. controls (4,-1.2) and (2,-2) .. (0,-2) .. controls (-2,-2) and (-4,-1.2) .. (-4,0);
     \draw (-2,0.2) arc (-120:-60:4);
     \draw (1.36,-0.08) arc (70:110:4);
\draw [very thick, postaction={decorate}] (0,-1.1) .. controls (3.5,-1.1) and (3.5,1.1) .. (0,1.1);
\draw [very thick] (0,-1.1) .. controls (-3.5,-1.1) and (-3.5,1.1) .. (0,1.1);
\node at (2,-1.2) {\tiny $\lambda$};
\draw [very thick, postaction={decorate}] (.5,-1.05) .. controls (.5,-1.5) and (.2,-2) .. (0,-2);
\draw [thick, dotted] (0,-2) .. controls (-.5,-2) and (-.5,-.35) .. (0,-.35);
\draw [very thick] (0,-.35) .. controls (.2,-.35) and (.5,-.6) .. (.5,-1.05);
\node at (0,-2.4) {\tiny $\mu$};
\node at (3.2,-.5) {$\circlearrowleft$};
\end{tikzpicture}
};
\endxy
\]
\caption{Conventions for $\tor$}\label{torusconventions}
\end{figure}

Up to isotopy, the connected embedded curves on $\tor$ are classified by pairs $(p,q)$ with $GCD(p,q)=1$ and $(p,q)\in \N^*\times \Z \cup \{0\}\times \N^*$. A pair $(p,q)$ corresponds to the only connected, embedded curve in $\tor$ that can be oriented such that its homology is the same as $p\lambda+q\mu$. Figure \ref{torcurveex} has an example. 

\begin{figure}[h]
\[
\xy (0,0)*{
  \begin{tikzpicture} [scale=1,fill opacity=0.2,  decoration={markings, 
                        mark=at position 0.6 with {\arrow{>}};    }]
\draw (-2,0) .. controls (-2,.6) and (-1,1) .. (0,1) .. controls (1,1) and (2,.6) .. (2,0) .. controls (2,-.6) and (1,-1) .. (0,-1) .. controls (-1,-1) and (-2,-.6) .. (-2,0);
     \draw (-1,0.1) arc (-120:-60:2);
     \draw (.68,-0.04) arc (70:110:2);
     \draw [very thick] (0,-1) .. controls (0.1,-1) and (.3,-.85) .. (.4,-.7) -- (.4,-.7).. controls (.5,-.5) and (1.7,-.3) .. (1.7,0) -- (1.7,0).. controls (1.7,.8) and (-1.7,.8) .. (-1.7,0) -- (-1.7,0) .. controls (-1.7,-.3)  and (-1.2,-.5) .. (-1,-.5) .. controls (-.6,-.5) and (-.15,-.16) .. (-.1,-.16);
     \draw [dotted] (-.1,-.16) .. controls (-.05,-.16) and (-.05,-1) .. (0,-1);
    \end{tikzpicture}};
\endxy
\]
\caption{A curve of type $(1,-1)$}\label{torcurveex}
\end{figure}
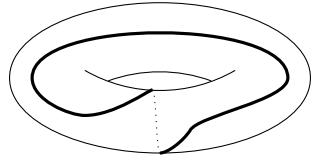

Denote by $(np,nq)$ the multicurve formed by $n$ parallel copies of the $(p,q)$ curve. Then the following set together with the empty curve forms a basis for $Sk(\tor)$.

$$\{(np,nq)\;|\;GCD(p,q)=1,\; (p,q)\in(\Z^{+*}\times \Z) \cup (\{0\}\times \Z^{+*}),\; n\in \Z^{+*}\}$$

Note that the number of intersections between two multicurves corresponding to pairs $(np,nq)$ and $(mr,ms)$ placed in generic position is given by the absolute value of the following determinant:
\[
\begin{vmatrix}
np & mr \\
nq & ms
\end{vmatrix}
= nm\begin{vmatrix}
p & r \\
q & s
\end{vmatrix}.
\]

When doing calculations involving multiplication of multicurves, we will always assume they are in generic position. Thus the number of intersections is described by the above formula, the top curve is consistently above the bottom curve in all crossings, and there are no crossings between a curve and itself.

%
\section{Frohman-Gelca formula}
%

The multiplicative structure on the Kauffman Bracket skein algebra was somewhat mysterious when first introduced. As a way to better understand the torus case, Frohman and Gelca characterize multiplication in $Sk(\tor)$ via an isomorphism with the noncommutative torus (see also Sallenave's independent work \cite{Sal1} for a different approach to the same of the same question). As part of this work, they define an alternative basis for $Sk(\tor)$ using Chebyshev polynomials of the first kind. Multiplication on this basis can be described by a simple product-to-sum formula.

\emph{Chebyshev polynomials of the first kind} are defined recursively by: $T_0=2$, $T_1=X$, $T_n=X\cdot T_{n-1}-T_{n-2}$.  For example, then, $T_2(X) = X^2-2$. The more common \emph{Chebyshev polynomials of the second kind}, which show up when defining Jones-Wenzl projectors, differ only from the first kind in that $T_0=1$.

Using the algebra structure, we can evaluate the polynomials $T_n$ on elements of $Sk(\tor)$. Indeed, given a simple closed curve $y\in Sk(\tor)$ and $k\in \N$, the expression $y^k$ represents a superposition of $k$ copies of the curve $y$. By interpreting constant terms as scalar multiples of the empty curve, we see that $T_n(y)\in Sk(\tor)$ for all $n$. Let $(p,q)$ be relatively prime integers representing a toric curve. Define $(np,nq)_T=T_n((p,q)) $. More generally, let $(a,b)_T$ be given by the following formula.
\[(a,b)_T=T_{GCD(a,b)}\left( \left(\frac{a}{GCD(a,b)},\frac{b}{GCD(a,b)}\right) \right) \]

\begin{example}

\begin{eqnarray}
 \quad
\left(
\xy (0,0)*{
  \begin{tikzpicture} [scale=.7,fill opacity=0.2,  decoration={markings, 
                        mark=at position 0.6 with {\arrow{>}};    }]
\draw (-2,0) .. controls (-2,.6) and (-1,1) .. (0,1) .. controls (1,1) and (2,.6) .. (2,0) .. controls (2,-.6) and (1,-1) .. (0,-1) .. controls (-1,-1) and (-2,-.6) .. (-2,0);
\draw (-1,0.1) arc (-120:-60:2);
\draw (.68,-0.04) arc (70:110:2);
     \draw [very thick] (0,-1) .. controls (0.1,-1) and (.3,-.85) .. (.4,-.7) -- (.4,-.7).. controls (.5,-.5) and (1.7,-.3) .. (1.7,0) -- (1.7,0).. controls (1.7,.8) and (-1.7,.8) .. (-1.7,0) -- (-1.7,0) .. controls (-1.7,-.3)  and (-1.2,-.5) .. (-1,-.5) .. controls (-.6,-.5) and (-.15,-.16) .. (-.1,-.16);
     \draw [dotted] (-.1,-.16) .. controls (-.05,-.16) and (-.05,-1) .. (0,-1);
     \draw [very thick] (.2,-1) .. controls (0.3,-1) and (.5,-.85) .. (.6,-.7) -- (.6,-.7).. controls (.7,-.5) and (1.9,-.3) .. (1.9,0) -- (1.9,0).. controls (1.9,1) and (-1.9,1) .. (-1.9,0) -- (-1.9,0) .. controls (-1.9,-.5)  and (-1.4,-.6) .. (-1.2,-.6) .. controls (-.4,-.7) and (.05,-.16) .. (.1,-.16);
     \draw [dotted] (.1,-.16) .. controls (.15,-.16) and (.15,-1) .. (.2,-1);
    \end{tikzpicture}};
\endxy
\right)_T
=
(2,-2)_T &=& T_2((1,-1))= (1,-1)^2-2 \nonumber \\
&=& 
\left(\;
\xy (0,0)*{
  \begin{tikzpicture} [scale=.7,fill opacity=0.2,  decoration={markings, 
                        mark=at position 0.6 with {\arrow{>}};    }]
\draw (-2,0) .. controls (-2,.6) and (-1,1) .. (0,1) .. controls (1,1) and (2,.6) .. (2,0) .. controls (2,-.6) and (1,-1) .. (0,-1) .. controls (-1,-1) and (-2,-.6) .. (-2,0);
\draw (-1,0.1) arc (-120:-60:2);
\draw (.68,-0.04) arc (70:110:2);
     \draw [very thick] (0,-1) .. controls (0.1,-1) and (.3,-.85) .. (.4,-.7) -- (.4,-.7).. controls (.5,-.5) and (1.7,-.3) .. (1.7,0) -- (1.7,0).. controls (1.7,.8) and (-1.7,.8) .. (-1.7,0) -- (-1.7,0) .. controls (-1.7,-.3)  and (-1.2,-.5) .. (-1,-.5) .. controls (-.6,-.5) and (-.15,-.16) .. (-.1,-.16);
     \draw [dotted] (-.1,-.16) .. controls (-.05,-.16) and (-.05,-1) .. (0,-1);
    \end{tikzpicture}};
\endxy
\; \right)^2
\quad - \quad
2\;\;
\xy (0,0)*{
  \begin{tikzpicture} [scale=.7,fill opacity=0.2,  decoration={markings, 
                        mark=at position 0.6 with {\arrow{>}};    }]
\draw (-2,0) .. controls (-2,.6) and (-1,1) .. (0,1) .. controls (1,1) and (2,.6) .. (2,0) .. controls (2,-.6) and (1,-1) .. (0,-1) .. controls (-1,-1) and (-2,-.6) .. (-2,0);
\draw (-1,0.1) arc (-120:-60:2);
\draw (.68,-0.04) arc (70:110:2);
    \end{tikzpicture}};
\endxy
\nonumber \\
&=&
\xy (0,0)*{
  \begin{tikzpicture} [scale=.7,fill opacity=0.2,  decoration={markings, 
                        mark=at position 0.6 with {\arrow{>}};    }]
\draw (-2,0) .. controls (-2,.6) and (-1,1) .. (0,1) .. controls (1,1) and (2,.6) .. (2,0) .. controls (2,-.6) and (1,-1) .. (0,-1) .. controls (-1,-1) and (-2,-.6) .. (-2,0);
\draw (-1,0.1) arc (-120:-60:2);
\draw (.68,-0.04) arc (70:110:2);
     \draw [very thick] (0,-1) .. controls (0.1,-1) and (.3,-.85) .. (.4,-.7) -- (.4,-.7).. controls (.5,-.5) and (1.7,-.3) .. (1.7,0) -- (1.7,0).. controls (1.7,.8) and (-1.7,.8) .. (-1.7,0) -- (-1.7,0) .. controls (-1.7,-.3)  and (-1.2,-.5) .. (-1,-.5) .. controls (-.6,-.5) and (-.15,-.16) .. (-.1,-.16);
     \draw [dotted] (-.1,-.16) .. controls (-.05,-.16) and (-.05,-1) .. (0,-1);
     \draw [very thick] (.2,-1) .. controls (0.3,-1) and (.5,-.85) .. (.6,-.7) -- (.6,-.7).. controls (.7,-.5) and (1.9,-.3) .. (1.9,0) -- (1.9,0).. controls (1.9,1) and (-1.9,1) .. (-1.9,0) -- (-1.9,0) .. controls (-1.9,-.5)  and (-1.4,-.6) .. (-1.2,-.6) .. controls (-.4,-.7) and (.05,-.16) .. (.1,-.16);
     \draw [dotted] (.1,-.16) .. controls (.15,-.16) and (.15,-1) .. (.2,-1);
    \end{tikzpicture}};
\endxy
\quad - \quad
2\;\;
\xy (0,0)*{
  \begin{tikzpicture} [scale=.7,fill opacity=0.2,  decoration={markings, 
                        mark=at position 0.6 with {\arrow{>}};    }]
\draw (-2,0) .. controls (-2,.6) and (-1,1) .. (0,1) .. controls (1,1) and (2,.6) .. (2,0) .. controls (2,-.6) and (1,-1) .. (0,-1) .. controls (-1,-1) and (-2,-.6) .. (-2,0);
\draw (-1,0.1) arc (-120:-60:2);
\draw (.68,-0.04) arc (70:110:2);
    \end{tikzpicture}};
\endxy
\nonumber 
\end{eqnarray}

\end{example}

\vspace{.5cm}

A central result of Frohman-Gelca is the following theorem \cite{FG}.

\begin{theorem} \label{theorem::FG_Formula}
The elements $(np,nq)_{T}$ together with the empty curve form a basis\footnote{Note that if one wants to keep a basis over $\Z[A,A^{-1}]$, the only homotopically trivial generator should be chosen to be $1$, while the corresponding Chebyshev polynomial is $2$. Then the Frohman-Gelca formula does not strictly give a decomposition for all products, since the multiplication by $1$ is not covered. The definition of it should however be easy to guess! Another approach could be to invert $2$ in the base ring, in order to have an actual basis given by the Chebyshev polynomials.} of $Sk(\tor)$. Furthermore, multiplication on this basis is described by the following \emph{Product-To-Sum} formula:
\begin{equation} \label{FGFormula}
(a,b)_T\ast (c,d)_T=A^{{ }^{\begin{vmatrix}
a & c \\
b & d
\end{vmatrix}}}(a-c,b-d)_T+A^{{}^{-\begin{vmatrix}
a & c \\
b & d
\end{vmatrix}}}(a+c,b+d).
\end{equation}
\end{theorem}

\vspace{1cm}

The purpose this paper is to reprove Formula \ref{FGFormula}. Its original proof~\cite{FG} is based on an iteration, which makes the appearance of Chebyshev polynomials somewhat mysterious. Our work  focuses on highlighting the role of these polynomials by introducing an intermediate algebra $\Aa$ which is the symmetric part of an oriented skein module defined in Section \ref{subsection:Aa}.

%
\section{An oriented skein module} \label{subsection:Aa}
%

The oriented skein module we define here is used to give a new proof of Formula \ref{FGFormula}. Define $\Aa$ to be the $R$-module generated by oriented embedded curves on $\tor$ modulo isotopy and the following relations.

\begin{equation}
\xy
 (0,0)*{
\begin{tikzpicture} [decoration={markings,
                         mark=at position 0.6 with {\arrow{>}};    }]
\draw[dashed] (.75,0) arc(0:360:.75);
\draw[very thick,->] (.5,0) arc(0:360:.5);
\end{tikzpicture}};
\endxy
\;\; =\;-A^2 \;\;
\xy
 (0,0)*{
\begin{tikzpicture} [decoration={markings,
                         mark=at position 0.6 with {\arrow{>}};    }]
\draw[dashed] (.75,0) arc(0:360:.75);
\end{tikzpicture}};
\endxy
\quad , \quad
\xy
 (0,0)*{
\begin{tikzpicture} [decoration={markings,
                         mark=at position 0.6 with {\arrow{>}};    }]
\draw[dashed] (.75,0) arc(0:360:.75);
\draw[very thick,<-] (.5,0) arc(0:360:.5);
\end{tikzpicture}};
\endxy
\;\; =\;-A^{-2} \;\;
\xy
 (0,0)*{
\begin{tikzpicture} [decoration={markings,
                         mark=at position 0.6 with {\arrow{>}};    }]
\draw[dashed] (.75,0) arc(0:360:.75);
\end{tikzpicture}};
\endxy
\label{relA_circ}
\end{equation}

\begin{equation}
\xy
(0,0)*{
\begin{tikzpicture}[baseline=0cm, scale=0.6]
\draw[style=very thick, <-] (0,0) -- (0,2);
\draw[style=very thick, ->](1,0)--(1,2);
\draw[radius=1.2, style=dashed](.5,1)circle;
\end{tikzpicture}};
\endxy
\quad = \quad
\xy
(0,0)*{
\begin{tikzpicture}[baseline=0cm, scale =.6]
\draw[radius=1.2, style=dashed](4.5,1)circle;
\draw [style=very thick, ->](4,2) arc (180:360: .5cm);
\draw [style=very thick, <-](4,0) arc (180:0: .5cm);
\draw [style=very thick, <-](4.8,1) arc (0:180: .3cm);
\draw [style=very thick](4.8,1) arc (360:180: .3cm);
\end{tikzpicture}};
\endxy
\quad,\quad
\xy
(0,0)*{
\begin{tikzpicture}[baseline=0cm, scale=0.6]
\draw[style=very thick, ->] (0,0) -- (0,2);
\draw[style=very thick, <-](1,0)--(1,2);
\draw[radius=1.2, style=dashed](.5,1)circle;
\end{tikzpicture}};
\endxy
\quad = \quad
\xy
(0,0)*{
\begin{tikzpicture}[baseline=0cm, scale =.6]
\draw[radius=1.2, style=dashed](4.5,1)circle;
\draw [style=very thick, <-](4,2) arc (180:360: .5cm);
\draw [style=very thick, ->](4,0) arc (180:0: .5cm);
\draw [style=very thick, ->](4.8,1) arc (0:180: .3cm);
\draw [style=very thick](4.8,1) arc (360:180: .3cm);
\end{tikzpicture}};
\endxy
\label{relA_Reid}
\end{equation}

A key implication of Relations \ref{relA_Reid} is that parallel non-trivial curves with opposite orientations cancel in $\Aa$. Indeed, applying the appropriate relation yields two homotopically trivial and oppositely oriented circles; using Relations \ref{relA_circ}, the circles can be removed, and their coefficients multiply to $(-A^2)\cdot( -A^{-2}) = 1$. In effect, parallel curves with opposite orientations are inverse elements  in $\Aa$. This will be a key idea in understanding the role of the Chebyshev polynomials. 

For a given multi-curve in $\Aa$, Relations \ref{relA_circ} enable us to delete all trivial components. What remains is a collection of parallel, oriented copies of some $(p,q)$ curve not all of which are necessarily oriented in the same direction. As previously discussed, Relations \ref{relA_Reid} let us delete adjacent, oppositely oriented curves. Any multi-curve may therefore be reduced to parallel, identically oriented copies of a $(p,q)$ curve with a coefficient in $R$. Note that the topology of the torus plays a role here: non-trivial curves can only be parallel copies of the same connected non-trivial curve, which is not true for other surfaces.

Let $\gamma_{(np,nq)}$ denote the element of $\Aa$ given by $n$ parallel copies of the oriented curve $(p,q)$ where the orientation is chosen so that its homology class is $p\lambda + q\mu$. Then $\{\gamma_{(np,nq)}\;|\;GCD(p,q)=1,\;n>0,\;(p,q)\in \Z\times \Z \setminus (0,0)\}$ together with the empty curve constitutes a generating set for the $R$-module $\Aa$. Note the allowable values for $(p,q)$ have expanded from the unoriented case to account for both orientations of toric curves.

\begin{proposition} \label{prop:basisAa}
 $\{\gamma_{(np,nq)}\;|\;GCD(p,q)=1,n>0,\;(p,q)\in \Z\times \Z \setminus (0,0)\}\cup\{\emptyset\}$ is a basis of $\Aa$.
\end{proposition}

\noindent \begin{proof}
By the argument above,  $\{\gamma_{(np,nq)}\;|\;GCD(p,q)=1,\;n>0,\;(p,q)\in \Z\times \Z \setminus (0,0)\}\cup\{\emptyset\}$ generates $\Aa$. We now wish to prove that if there is a linear combination $\sum_{i}\alpha \gamma_i=0$ with $\alpha_i\in R$ and $\gamma_i$ in the set of generators, we must have all $\alpha_i=0$.

Let us first introduce a technical tool. Assume that an orthonormal vector field $(\ai,\aj)$ has been fixed on the torus (it can be determined by $\lambda$ and $-\mu$ for example, which we choose here). If $\gamma$ is connected and if the curve representing it is smooth, given by $s$ a function from $\Ss^1$ to the torus, then $t\rightarrow \frac{(Ts(t)\cdot \ai,Ts(t)\cdot \aj)}{||(Ts(t)\cdot \ai,Ts(t)\cdot \aj)||}$ gives a continuous function $\Ss^1\rightarrow \Ss^1$, where $Ts(t)$ is the tangent vector and the dot denotes the scalar product (this is called the \emph{Gauss map}). In homology, the function $\Ss^1\rightarrow \Ss^1$ we get is the multiplication by an integer $\delta(\gamma)$. This process can be extended linearly to non-connected curves. For connected curves with non-trivial homology, the Gauss integer is zero, while it is $\pm 1$ for trivial circles.

Relation \ref{relA_Reid} preserves this integer. In Relation \ref{relA_circ}, the sum of half the power of $A$ (so $\pm 1$) and of this integer is preserved (since the deletion of a circle is balanced by the multiplication by $-A^{\pm 2}$). This implies that the sum of half the power of $A$ and of the index gives a homogeneous decomposition of elements with respect to the relations, and hence we can decompose $\sum_i \alpha_i \gamma_i$ as follows:
\[
\sum_i\alpha_i\gamma_i=0=\sum_kA^k(\sum_ib_i\gamma_i)\;\text{with}\; b_i\in \Z \; \Rightarrow
\forall k\in \Z, \; \sum_i b_i \gamma_i=0.
\]

Now, observe that setting $A^2=-1$ induces a map: $\Aa^{\Theta}\mapsto H_1$ where $H_1$ is the homology of the torus. $\sum_i b_i \gamma_i=0$ is sent under this map to the same equation holding in homology. But the generators $\gamma$ form a basis of the homology, and therefore we have $b_i=0$ for all $i$. Doing this for all value of $k$, we finally get:
\[\sum_i\alpha_i\gamma_i=0 \; \Rightarrow \;\forall i,\;\alpha_i=0,\]
which proves that we have a basis of the module $\Aa^{\Theta}$.

\end{proof}

\begin{corollary}
 $\Aa$ is a torsion-free module.
\end{corollary}

We define a multiplication on $\Aa$ via superposition and oriented smoothing as depicted in \ref{def:AaSmooth}.

\begin{equation} \label{def:AaSmooth}
\xy
 (0,0)*{
\begin{tikzpicture} [decoration={markings,
                         mark=at position 1 with {\arrow{>}};    }]
\draw[dashed] (.75,0) arc(0:360:.75);
\draw[very thick, postaction={decorate}] (.5,-.5) -- (-.5,.5);
\draw[white, double, double distance=1.25pt, double=black, ultra thick] (-.5,-.5) -- (.4,.4);
\draw[very thick, postaction={decorate}] (.4,.4) -- (.5,.5); 
\end{tikzpicture}};
\endxy
\;\; =\;\; A\;\;
\xy
 (0,0)*{
\begin{tikzpicture}
\draw[dashed] (.75,0) arc(0:360:.75);
\draw[very thick, ->] (-.5,-.5) .. controls (0,0) .. (-.5,.5);
\draw[very thick, ->] (.5,-.5) .. controls (0,0) .. (.5,.5); 
\end{tikzpicture}};
\endxy
\quad, \quad
\xy
 (0,0)*{
\begin{tikzpicture} [decoration={markings,
                         mark=at position 1 with {\arrow{>}};    }]
\draw[dashed] (.75,0) arc(0:360:.75);
\draw[very thick, postaction={decorate}] (-.5,.5) -- (.5,-.5);
\draw[white, double, double distance=1.25pt, double=black, ultra thick] (-.5,-.5) -- (.4,.4);
\draw[very thick, postaction={decorate}] (.4,.4) -- (.5,.5); 
\end{tikzpicture}};
\endxy
\;\; =\;\; A^{-1}\;\;
\xy
 (0,0)*{
\begin{tikzpicture} [decoration={markings,
                         mark=at position 0.6 with {\arrow{>}};    }]
\draw[dashed] (.75,0) arc(0:360:.75);
\draw[very thick,->] (-.5,.5) .. controls (0,0) .. (.5,.5);
\draw[very thick,->] (-.5,-.5) .. controls (0,0) .. (.5,-.5); 
\end{tikzpicture}};
\endxy
\end{equation}

\begin{proposition}\label{Amult}
 The multiplication process is well-defined.
\end{proposition}
\begin{proof}
Isotopy of curves before superposition corresponds after superposition to oriented Reidemeister II moves. For each possible orientation of a Reidemeister II move, the oriented smoothings of the two sides are equal in $\Aa$. Depending on orientation, this can be seen in the following calculations.

\[
\xy
(0,0)*{
\begin{tikzpicture}[baseline=0cm, scale=0.6]
\draw[style=very thick, ->] (0,0) -- (0,2);
\draw[style=very thick, ->](1,0)--(1,2);
\draw[radius=1.2, style=dashed](.5,1)circle;
\end{tikzpicture}};
\endxy
\quad = \quad A \cdot A^{-1} \;
\xy
(0,0)*{
\begin{tikzpicture}[baseline=0cm, scale=0.6]
\draw[style=very thick, ->] (0,0) .. controls (.5,.5) and (0,.8) .. (0,1) .. controls (0,1.2) and (.5,1.5) .. (0,2);
\draw[style=very thick, ->](1,0) .. controls (.5,.5) and (1,.8) .. (1,1) .. controls (1,1.2) and (.5,1.5) .. (1,2);
\draw[radius=1.2, style=dashed](.5,1)circle;
\end{tikzpicture}};
\endxy
\quad = \quad
\xy
(0,0)*{
\begin{tikzpicture}[baseline=0cm, scale=0.6]
\draw[style=very thick, ->] (0,0) .. controls (1,1) .. (0,2);
\draw [style= very thick] (1,0) .. controls (.8,.15) ..  (.6,.35);
\draw [style = very thick] (.4,.55) .. controls (.1,.85) and (.1,1.15) .. (.4,1.45);
\draw[style=very thick, ->](.6,1.65) .. controls (.8,1.85) .. (1,2);
\draw[radius=1.2, style=dashed](.5,1)circle;
\end{tikzpicture}};
\endxy 
\]

\vspace{.1in}

\[
\xy
(0,0)*{
\begin{tikzpicture}[baseline=0cm, scale=0.6]
\draw[style=very thick, <-] (0,0) -- (0,2);
\draw[style=very thick, ->](1,0)--(1,2);
\draw[radius=1.2, style=dashed](.5,1)circle;
\end{tikzpicture}};
\endxy
\quad  \stackrel{3.0.7}{=} \quad
\xy
(0,0)*{
\begin{tikzpicture}[baseline=0cm, scale =.6]
\draw[radius=1.2, style=dashed](4.5,1)circle;
\draw [style=very thick, ->](4,2) arc (180:360: .5cm);
\draw [style=very thick, <-](4,0) arc (180:0: .5cm);
\draw [style=very thick, <-](4.8,1) arc (0:180: .3cm);
\draw [style=very thick](4.8,1) arc (360:180: .3cm);
\end{tikzpicture}};
\endxy
\quad = \quad
\xy
(0,0)*{
\begin{tikzpicture}[baseline=0cm, scale=0.6]
\draw[style=very thick, <-] (0,0) .. controls (1,1) .. (0,2);
\draw [style= very thick] (1,0) .. controls (.8,.15) ..  (.6,.35);
\draw [style = very thick] (.4,.55) .. controls (.1,.85) and (.1,1.15) .. (.4,1.45);
\draw[style=very thick, ->](.6,1.65) .. controls (.8,1.85) .. (1,2);
\draw[radius=1.2, style=dashed](.5,1)circle;
\end{tikzpicture}};
\endxy \hspace{.35in}
\]

\end{proof}

As with unoriented links and the Kauffman Bracket skein module, there exists a map from oriented links in the thickened torus to $\Aa$ defined by projection followed by oriented smoothing. Proposition \ref{Amult} shows that Reidemeister II holds in this context. One can check that Reidemeister III also holds. However, even the framed version of Reidemeister I fails. It appears actually that the image in $\Aa$ of an oriented framed knot is exactly invariant under moves that keep the framing transverse to obvious vertical vector field. These moves can be proven to be RII and RIII by looking at the Gauss map associated with projections. Although the module $\Aa$ is not exactly a skein module for oriented knots in the thickened torus, viewing Reidemeister moves via relations in $\Aa$ is useful for relating $\Aa$ to knot invariants.

We conclude this section with a product-to-sum formula for multiplying basis elements in $\Aa$. This is key in establishing our main result.

\begin{proposition}
Let $\gamma_{a,b}, \gamma_{-a,-b}, \gamma_{c,d}$ and $\gamma_{-c,-d}$ be basis elements in $\Aa$. Then the following formula holds.
$$
(\gamma_{a,b}+\gamma_{-a,-b}) \ast (\gamma_{c,d}+\gamma_{-c,-d})= A^{\tiny{{ }^{\setlength{\arraycolsep}{1pt} \begin{vmatrix}
a & c \\
b & d
\end{vmatrix}}}} (\gamma_{a-c,b-d}+\gamma_{-a+c,-b+d}) 
 +A^{\tiny{{ }^{-\setlength{\arraycolsep}{1pt} \begin{vmatrix}
a & c \\
b & d
\end{vmatrix}}}} (\gamma_{a+c,b+d}+\gamma_{-a-c,-b-d})$$
\end{proposition}

\begin{proof}
The proof is somewhat immediate when visualized diagrammatically. On the left hand side, recall that each basis element consists of a collection of identically oriented parallel curves. Expanding the product, we have four terms each of the form $\gamma\ast\gamma'$  with $\gamma, \gamma'$ basis elements. Understanding these individual terms is easier than in the unoriented setting since there is a unique oriented smoothing for any product of curves. After smoothing crossings, then, the product still has only four terms. Moreover, these terms are two pairs of basis elements of $\Aa$ each pair consisting of the two possible orientations of the same underlying parallel multicurve. It is easy to identify the result of the smoothing of each of the products, in particular since the homotopy is preserved: we obtain curves $\gamma_{\pm (a+c,b+d)}$ and $\gamma_{\pm(a-c,b-d)}$.

Example \ref{ex:OrProdProof} should help illuminate the proof.

\end{proof}

\begin{example} \label{ex:OrProdProof}
To help illuminate this, we give the following example.
\begin{eqnarray}
\left(
\xy (0,0)*{
  \begin{tikzpicture} [scale=.3,  decoration={markings, 
                        mark=at position 0.25 with {\arrow{>}};    }]
\draw (-4,0) .. controls (-4,1.2) and (-2,2) .. (0,2) .. controls (2,2) and (4,1.2) .. (4,0) .. controls (4,-1.2) and (2,-2) .. (0,-2) .. controls (-2,-2) and (-4,-1.2) .. (-4,0);
\draw (-2,0.2) arc (-120:-60:4);
\draw (1.36,-0.08) arc (70:110:4);
\draw [very thick, postaction={decorate}] (2.4,0).. controls (2.4,1.2) and (-2.4,1.2) .. (-2.4,0) .. controls (-2.4,-1.2)  and (2.4,-1.2) .. (2.4,0);
\draw [very thick, postaction={decorate}] (3.2,0).. controls (3.2,2) and (-3.2,2) .. (-3.2,0) .. controls (-3.2,-2)  and (3.2,-2) .. (3.2,0);
    \end{tikzpicture}};
\endxy
+
\xy (0,0)*{
  \begin{tikzpicture} [scale=.3,  decoration={markings, 
                        mark=at position 0.25 with {\arrow{<}};    }]
\draw (-4,0) .. controls (-4,1.2) and (-2,2) .. (0,2) .. controls (2,2) and (4,1.2) .. (4,0) .. controls (4,-1.2) and (2,-2) .. (0,-2) .. controls (-2,-2) and (-4,-1.2) .. (-4,0);
\draw (-2,0.2) arc (-120:-60:4);
\draw (1.36,-0.08) arc (70:110:4);
\draw [very thick, postaction={decorate}] (2.4,0).. controls (2.4,1.2) and (-2.4,1.2) .. (-2.4,0) .. controls (-2.4,-1.2)  and (2.4,-1.2) .. (2.4,0);
\draw [very thick, postaction={decorate}] (3.2,0).. controls (3.2,2) and (-3.2,2) .. (-3.2,0) .. controls (-3.2,-2)  and (3.2,-2) .. (3.2,0);
    \end{tikzpicture}};
\endxy
\right)
&\ast&
\left(
\xy (0,0)*{
  \begin{tikzpicture} [scale=.3,  decoration={markings, 
                        mark=at position 0.5 with {\arrow{>}};    }]
\draw (-4,0) .. controls (-4,1.2) and (-2,2) .. (0,2) .. controls (2,2) and (4,1.2) .. (4,0) .. controls (4,-1.2) and (2,-2) .. (0,-2) .. controls (-2,-2) and (-4,-1.2) .. (-4,0);
\draw (-2,0.2) arc (-120:-60:4);
\draw (1.36,-0.08) arc (70:110:4);
\draw [very thick,postaction={decorate}] (.5,-2) .. controls (.9,-2) and (.7,-.3) ..(.3,-.3);
\draw [very thick,postaction={decorate}] (1.3,-1.9) .. controls (1.7,-1.9) and (1.4,-.2) ..(1.1,-.2);
 \draw [dotted, very thick] (.5,-2) .. controls (.1,-2) and (.1,-.3) .. (.3,-.3);
 \draw [dotted, very thick] (1.3,-1.9) .. controls (1.1,-1.9) and (.9,-.2) .. (1.1,-.2);
    \end{tikzpicture}};
\endxy
+
\xy (0,0)*{
  \begin{tikzpicture} [scale=.3,  decoration={markings, 
                        mark=at position 0.5 with {\arrow{<}};    }]
\draw (-4,0) .. controls (-4,1.2) and (-2,2) .. (0,2) .. controls (2,2) and (4,1.2) .. (4,0) .. controls (4,-1.2) and (2,-2) .. (0,-2) .. controls (-2,-2) and (-4,-1.2) .. (-4,0);
\draw (-2,0.2) arc (-120:-60:4);
\draw (1.36,-0.08) arc (70:110:4);
\draw [very thick,postaction={decorate}] (.5,-2) .. controls (.9,-2) and (.7,-.3) ..(.3,-.3);
\draw [very thick,postaction={decorate}] (1.3,-1.9) .. controls (1.7,-1.9) and (1.4,-.2) ..(1.1,-.2);
 \draw [dotted, very thick] (.5,-2) .. controls (.1,-2) and (.1,-.3) .. (.3,-.3);
 \draw [dotted, very thick] (1.3,-1.9) .. controls (1.1,-1.9) and (.9,-.2) .. (1.1,-.2);
    \end{tikzpicture}};
\endxy
\right)
\nonumber \\
&=&
A^4
\left(
\xy (0,0)*{
  \begin{tikzpicture} [scale=.3,  decoration={markings, 
                        mark=at position 0.25 with {\arrow{>}};    }]
\draw (-4,0) .. controls (-4,1.2) and (-2,2) .. (0,2) .. controls (2,2) and (4,1.2) .. (4,0) .. controls (4,-1.2) and (2,-2) .. (0,-2) .. controls (-2,-2) and (-4,-1.2) .. (-4,0);
\draw (-2,0.2) arc (-120:-60:4);
\draw (1.36,-0.08) arc (70:110:4);
\draw [very thick] (.5,-2) .. controls (.9,-2) and (.7,-.3) ..(.3,-.3);
\draw [very thick] (1.3,-1.9) .. controls (1.7,-1.9) and (1.4,-.2) ..(1.1,-.2);
 \draw [dotted, very thick] (.5,-2) .. controls (.1,-2) and (.1,-.3) .. (.3,-.3);
 \draw [dotted, very thick] (1.3,-1.9) .. controls (1.1,-1.9) and (.9,-.2) .. (1.1,-.2);
\draw [very thick,postaction={decorate}] (2.4,0).. controls (2.4,1.2) and (-2.4,1.2) .. (-2.4,0) .. controls (-2.4,-1.2)  and (2.4,-1.2) .. (2.4,0);
\draw [very thick,postaction={decorate}] (3.2,0).. controls (3.2,2) and (-3.2,2) .. (-3.2,0) .. controls (-3.2,-2)  and (3.2,-2) .. (3.2,0);
\fill[fill=white] (.83,-.85) arc (0:360:.2);
\draw [very thick] (.83,-.85) .. controls (.73,-.85) and (.63,-.95) .. (.69,-1.05);
\draw [very thick] (.42,-.89) .. controls (.53,-.85) and (.63,-.75) .. (.58,-.63);
\fill[fill=white] (1.61,-.75) arc (0:360:.2);
\draw [very thick] (1.61,-.69) .. controls (1.51,-.75) and (1.41,-.85) .. (1.45,-.95);
\draw [very thick] (1.2,-.79) .. controls (1.31,-.75) and (1.41,-.65) .. (1.34,-.53);
\fill[fill=white] (.9,-1.45) arc (0:360:.2);
\draw [very thick] (.91,-1.45) .. controls (.8,-1.45) and (.7,-1.55) .. (.71,-1.66);
\draw [very thick] (.5,-1.48) .. controls (.6,-1.48) and (.71,-1.4) .. (.71,-1.24);
\fill[fill=white] (1.7,-1.33) arc (0:360:.2);
\draw [very thick] (1.7,-1.3) .. controls (1.6,-1.3) and (1.49,-1.43) .. (1.49,-1.53);
\draw [very thick] (1.3,-1.38) .. controls (1.4,-1.38) and (1.47,-1.23) .. (1.47,-1.13);
\end{tikzpicture}};
\endxy
+\xy (0,0)*{
  \begin{tikzpicture} [scale=.3,  decoration={markings, 
                        mark=at position 0.25 with {\arrow{<}};    }]
\draw (-4,0) .. controls (-4,1.2) and (-2,2) .. (0,2) .. controls (2,2) and (4,1.2) .. (4,0) .. controls (4,-1.2) and (2,-2) .. (0,-2) .. controls (-2,-2) and (-4,-1.2) .. (-4,0);
\draw (-2,0.2) arc (-120:-60:4);
\draw (1.36,-0.08) arc (70:110:4);
\draw [very thick] (.5,-2) .. controls (.9,-2) and (.7,-.3) ..(.3,-.3);
\draw [very thick] (1.3,-1.9) .. controls (1.7,-1.9) and (1.4,-.2) ..(1.1,-.2);
 \draw [dotted, very thick] (.5,-2) .. controls (.1,-2) and (.1,-.3) .. (.3,-.3);
 \draw [dotted, very thick] (1.3,-1.9) .. controls (1.1,-1.9) and (.9,-.2) .. (1.1,-.2);
\draw [very thick,postaction={decorate}] (2.4,0).. controls (2.4,1.2) and (-2.4,1.2) .. (-2.4,0) .. controls (-2.4,-1.2)  and (2.4,-1.2) .. (2.4,0);
\draw [very thick,postaction={decorate}] (3.2,0).. controls (3.2,2) and (-3.2,2) .. (-3.2,0) .. controls (-3.2,-2)  and (3.2,-2) .. (3.2,0);
\fill[fill=white] (.83,-.85) arc (0:360:.2);
\draw [very thick] (.83,-.85) .. controls (.73,-.85) and (.63,-.95) .. (.69,-1.05);
\draw [very thick] (.42,-.89) .. controls (.53,-.85) and (.63,-.75) .. (.58,-.63);
\fill[fill=white] (1.61,-.75) arc (0:360:.2);
\draw [very thick] (1.61,-.69) .. controls (1.51,-.75) and (1.41,-.85) .. (1.45,-.95);
\draw [very thick] (1.2,-.79) .. controls (1.31,-.75) and (1.41,-.65) .. (1.34,-.53);
\fill[fill=white] (.9,-1.45) arc (0:360:.2);
\draw [very thick] (.91,-1.45) .. controls (.8,-1.45) and (.7,-1.55) .. (.71,-1.66);
\draw [very thick] (.5,-1.48) .. controls (.6,-1.48) and (.71,-1.4) .. (.71,-1.24);
\fill[fill=white] (1.7,-1.33) arc (0:360:.2);
\draw [very thick] (1.7,-1.3) .. controls (1.6,-1.3) and (1.49,-1.43) .. (1.49,-1.53);
\draw [very thick] (1.3,-1.38) .. controls (1.4,-1.38) and (1.47,-1.23) .. (1.47,-1.13);
\end{tikzpicture}};
\endxy
\right)
\nonumber \\
&+&A^{-4}
\left(
\xy (0,0)*{
  \begin{tikzpicture} [scale=.3,  decoration={markings, 
                        mark=at position 0.25 with {\arrow{>}};    }]
\draw (-4,0) .. controls (-4,1.2) and (-2,2) .. (0,2) .. controls (2,2) and (4,1.2) .. (4,0) .. controls (4,-1.2) and (2,-2) .. (0,-2) .. controls (-2,-2) and (-4,-1.2) .. (-4,0);
\draw (-2,0.2) arc (-120:-60:4);
\draw (1.36,-0.08) arc (70:110:4);
\draw [very thick] (.5,-2) .. controls (.9,-2) and (.7,-.3) ..(.3,-.3);
\draw [very thick] (1.3,-1.9) .. controls (1.7,-1.9) and (1.4,-.2) ..(1.1,-.2);
 \draw [dotted, very thick] (.5,-2) .. controls (.1,-2) and (.1,-.3) .. (.3,-.3);
 \draw [dotted, very thick] (1.3,-1.9) .. controls (1.1,-1.9) and (.9,-.2) .. (1.1,-.2);
\draw [very thick,postaction={decorate}] (2.4,0).. controls (2.4,1.2) and (-2.4,1.2) .. (-2.4,0) .. controls (-2.4,-1.2)  and (2.4,-1.2) .. (2.4,0);
\draw [very thick,postaction={decorate}] (3.2,0).. controls (3.2,2) and (-3.2,2) .. (-3.2,0) .. controls (-3.2,-2)  and (3.2,-2) .. (3.2,0);
\fill[fill=white] (.83,-.85) arc (0:360:.2);
\draw [very thick] (.83,-.85) .. controls (.73,-.85) and (.63,-.75) .. (.58,-.63);
\draw [very thick] (.42,-.89) .. controls (.53,-.85) and (.63,-.95) .. (.69,-1.05);
\fill[fill=white] (1.61,-.75) arc (0:360:.2);
\draw [very thick] (1.61,-.69) .. controls (1.51,-.75) and (1.41,-.65) .. (1.34,-.53);
\draw [very thick] (1.2,-.79) .. controls (1.31,-.75) and (1.41,-.85) .. (1.45,-.95);
\fill[fill=white] (.9,-1.45) arc (0:360:.2);
\draw [very thick] (.91,-1.45) .. controls (.8,-1.45) and (.71,-1.4) .. (.71,-1.24);
\draw [very thick] (.5,-1.48) .. controls (.6,-1.48) and (.7,-1.55) .. (.71,-1.66);
\fill[fill=white] (1.7,-1.33) arc (0:360:.2);
\draw [very thick] (1.7,-1.3) .. controls (1.6,-1.3)  and (1.47,-1.23) .. (1.47,-1.13);
\draw [very thick] (1.3,-1.38) .. controls (1.4,-1.38)and (1.49,-1.43) .. (1.49,-1.53);
\end{tikzpicture}};
\endxy
+\xy (0,0)*{
  \begin{tikzpicture} [scale=.3,  decoration={markings, 
                        mark=at position 0.25 with {\arrow{<}};    }]
\draw (-4,0) .. controls (-4,1.2) and (-2,2) .. (0,2) .. controls (2,2) and (4,1.2) .. (4,0) .. controls (4,-1.2) and (2,-2) .. (0,-2) .. controls (-2,-2) and (-4,-1.2) .. (-4,0);
\draw (-2,0.2) arc (-120:-60:4);
\draw (1.36,-0.08) arc (70:110:4);
\draw [very thick] (.5,-2) .. controls (.9,-2) and (.7,-.3) ..(.3,-.3);
\draw [very thick] (1.3,-1.9) .. controls (1.7,-1.9) and (1.4,-.2) ..(1.1,-.2);
 \draw [dotted, very thick] (.5,-2) .. controls (.1,-2) and (.1,-.3) .. (.3,-.3);
 \draw [dotted, very thick] (1.3,-1.9) .. controls (1.1,-1.9) and (.9,-.2) .. (1.1,-.2);
\draw [very thick,postaction={decorate}] (2.4,0).. controls (2.4,1.2) and (-2.4,1.2) .. (-2.4,0) .. controls (-2.4,-1.2)  and (2.4,-1.2) .. (2.4,0);
\draw [very thick,postaction={decorate}] (3.2,0).. controls (3.2,2) and (-3.2,2) .. (-3.2,0) .. controls (-3.2,-2)  and (3.2,-2) .. (3.2,0);
\fill[fill=white] (.83,-.85) arc (0:360:.2);
\draw [very thick] (.83,-.85) .. controls (.73,-.85) and (.63,-.75) .. (.58,-.63);
\draw [very thick] (.42,-.89) .. controls (.53,-.85) and (.63,-.95) .. (.69,-1.05);
\fill[fill=white] (1.61,-.75) arc (0:360:.2);
\draw [very thick] (1.61,-.69) .. controls (1.51,-.75) and (1.41,-.65) .. (1.34,-.53);
\draw [very thick] (1.2,-.79) .. controls (1.31,-.75) and (1.41,-.85) .. (1.45,-.95);
\fill[fill=white] (.9,-1.45) arc (0:360:.2);
\draw [very thick] (.91,-1.45) .. controls (.8,-1.45) and (.71,-1.4) .. (.71,-1.24);
\draw [very thick] (.5,-1.48) .. controls (.6,-1.48) and (.7,-1.55) .. (.71,-1.66);
\fill[fill=white] (1.7,-1.33) arc (0:360:.2);
\draw [very thick] (1.7,-1.3) .. controls (1.6,-1.3)  and (1.47,-1.23) .. (1.47,-1.13);
\draw [very thick] (1.3,-1.38) .. controls (1.4,-1.38)and (1.49,-1.43) .. (1.49,-1.53);
\end{tikzpicture}};
\endxy
\right)
\nonumber
\end{eqnarray}

One of the intermediate terms is as follows and the three other ones are similar.
\begin{equation}
\xy (0,0)*{
  \begin{tikzpicture} [scale=.5,  decoration={markings, 
                        mark=at position 0.95 with {\arrow{>}};    }]
\draw (-4,0) .. controls (-4,1.2) and (-2,2) .. (0,2) .. controls (2,2) and (4,1.2) .. (4,0) .. controls (4,-1.2) and (2,-2) .. (0,-2) .. controls (-2,-2) and (-4,-1.2) .. (-4,0);
\draw (-2,0.2) arc (-120:-60:4);
\draw (1.36,-0.08) arc (70:110:4);
\draw [very thick,postaction={decorate}] (.5,-2) .. controls (.9,-2) and (.7,-.3) ..(.3,-.3);
\draw [very thick,postaction={decorate}] (1.3,-1.9) .. controls (1.7,-1.9) and (1.4,-.2) ..(1.1,-.2);
 \draw [dotted, very thick] (.5,-2) .. controls (.1,-2) and (.1,-.3) .. (.3,-.3);
 \draw [dotted, very thick] (1.3,-1.9) .. controls (1.1,-1.9) and (.9,-.2) .. (1.1,-.2);
\draw [thick, double, color= white, double=black,double distance=1.25pt] (2.4,0).. controls (2.4,1.2) and (-2.4,1.2) .. (-2.4,0) .. controls (-2.4,-1.2)  and (2.4,-1.2) .. (2.4,0);
\draw [thick, double, color= white, double=black,double distance=1.25pt] (3.2,0).. controls (3.2,2) and (-3.2,2) .. (-3.2,0) .. controls (-3.2,-2)  and (3.2,-2) .. (3.2,0);
\draw [very thick, postaction={decorate}] (.1,1.48)-- (0,1.48);
\draw [very thick, postaction={decorate}] (.1,.88)-- (0,.88);
    \end{tikzpicture}};
\endxy
\quad = \quad
A^4
\;
\xy (0,0)*{
  \begin{tikzpicture} [scale=.5,  decoration={markings, 
                        mark=at position 0.25 with {\arrow{>}};    }]
\draw (-4,0) .. controls (-4,1.2) and (-2,2) .. (0,2) .. controls (2,2) and (4,1.2) .. (4,0) .. controls (4,-1.2) and (2,-2) .. (0,-2) .. controls (-2,-2) and (-4,-1.2) .. (-4,0);
\draw (-2,0.2) arc (-120:-60:4);
\draw (1.36,-0.08) arc (70:110:4);
\draw [very thick] (.5,-2) .. controls (.9,-2) and (.7,-.3) ..(.3,-.3);
\draw [very thick] (1.3,-1.9) .. controls (1.7,-1.9) and (1.4,-.2) ..(1.1,-.2);
 \draw [dotted, very thick] (.5,-2) .. controls (.1,-2) and (.1,-.3) .. (.3,-.3);
 \draw [dotted, very thick] (1.3,-1.9) .. controls (1.1,-1.9) and (.9,-.2) .. (1.1,-.2);
\draw [very thick,postaction={decorate}] (2.4,0).. controls (2.4,1.2) and (-2.4,1.2) .. (-2.4,0) .. controls (-2.4,-1.2)  and (2.4,-1.2) .. (2.4,0);
\draw [very thick,postaction={decorate}] (3.2,0).. controls (3.2,2) and (-3.2,2) .. (-3.2,0) .. controls (-3.2,-2)  and (3.2,-2) .. (3.2,0);
\fill[fill=white] (.83,-.85) arc (0:360:.2);
\draw [very thick] (.83,-.85) .. controls (.73,-.85) and (.63,-.95) .. (.69,-1.05);
\draw [very thick] (.42,-.89) .. controls (.53,-.85) and (.63,-.75) .. (.58,-.63);
\fill[fill=white] (1.61,-.75) arc (0:360:.2);
\draw [very thick] (1.61,-.69) .. controls (1.51,-.75) and (1.41,-.85) .. (1.45,-.95);
\draw [very thick] (1.2,-.79) .. controls (1.31,-.75) and (1.41,-.65) .. (1.34,-.53);
\fill[fill=white] (.9,-1.45) arc (0:360:.2);
\draw [very thick] (.91,-1.45) .. controls (.8,-1.45) and (.7,-1.55) .. (.71,-1.66);
\draw [very thick] (.5,-1.48) .. controls (.6,-1.48) and (.71,-1.4) .. (.71,-1.24);
\fill[fill=white] (1.7,-1.33) arc (0:360:.2);
\draw [very thick] (1.7,-1.3) .. controls (1.6,-1.3) and (1.49,-1.43) .. (1.49,-1.53);
\draw [very thick] (1.3,-1.38) .. controls (1.4,-1.38) and (1.47,-1.23) .. (1.47,-1.13);
\end{tikzpicture}};
\endxy
\nonumber
\end{equation}
\end{example}

%
\section{Reproving Frohman-Gelca formula} \label{subsection:reproof}
%

%
%
\subsection{Technical result}

Denote by $\Theta$ the $R$-module map on $\Aa$ that reverses orientation of generators, so $\Theta(\gamma_{(np,nq)})=\gamma_{(-np,-nq)}$. For a general curve, this corresponds to reversing the orientations of non-trivial components. Denote by $\Aa^{\Theta}$ the part of $\Aa$ that is stable under $\Theta$. $\Theta$ is a map of algebras, and $\Aa^{\Theta}$ is a sub-algebra of $\Aa$, whose basis is given by the symmetrization of the original basis of $\Aa$: 
\[
\{\gamma_{(np,nq)}+\gamma_{(-np,-nq)}\;|\; n>0, (p,q)\in (\Z^{>0}\times \Z \cup \{0\}\times \Z^{>0})\}\cup \{1\}.
\]

Consider the map $\psi\ :\ Sk(\tor)\rightarrow \Aa^{\Theta}$ that to a simple multicurve $\gamma$ associates the sum of all possible orientations of this curve. We will prove that this natural map is in fact an algebra isomorphism in Theorem \ref{THM_iso} below. 
This result is the technical heart of the relation between $Sk(\tor)$ and our oriented version $\Aa$ and does not depend on the use of Chebyshev polynomials. The technicality of Frohman-Gelca's proof is shifted, in our context, to this theorem, while the role of the Chebyshev polynomials will appear entirely inside $\Aa$. Note that the following Theorem \ref{THM_iso} is very closely related to results of Sallenave \cite{MR1834496}.

\begin{theorem} \label{THM_iso}
$\psi: Sk(\tor)\rightarrow \Aa^{\Theta}$ is an isomorphism of $R$-algebras. 
\end{theorem}

\begin{proof}
It is straightforward to check that $\psi$ gives a well-defined map of $R$-modules, and this is an isomorphism. Indeed, the image of the usual basis of $Sk(\tor)$ is a basis for $\Aa^{\Theta}$. It remains to check that this gives a map of $R$-algebras. Let $\gamma_1$ and $\gamma_2$ be two basis elements of $Sk(\tor)$ which are in generic position. Our goal is to show $\psi(\gamma_1\ast \gamma_2)-\psi(\gamma_1)\psi(\gamma_2) = 0$  thus verifying that the map $\psi$ preserves the algebra structure. 

Say there are $k$ crossings in $\gamma_1\ast \gamma_2$. After performing unoriented smoothings of all crossings in $Sk(\tor)$, there are $2^k$ terms in $\gamma_1\ast \gamma_2$ . At each crossing, there are four possible local orientations, so there is a maximum of $4^k$ possible orientations for each term of the product $\gamma_1\ast \gamma_2$. However, not every choice of local orientation at each crossing corresponds to a valid global orientation, so the total number of possible orientations of a term of $\gamma_1\ast \gamma_2$ may be less than $4^k$.  (This happens, for instance, when there are two crossings between two connected components in a smoothing.) 

Before smoothing $\gamma_1\ast \gamma_2$, each crossing has two possible local orientations. Thus, there is a maximum of $2^k$ terms in $\psi(\gamma_1)\psi(\gamma_2)$. Again, not every possible collection of local orientations necessarily forms a valid global orientation, so this maximum may not be achieved. In $\Aa$, each oriented crossing has a unique smoothing, so, after smoothing, the product $\psi(\gamma_1)\psi(\gamma_2)$ still has a maximum of $2^k$ terms.

Isolate some crossing in $\gamma_1\ast \gamma_2$. Then the following figure summarizes the above discussion. While some local orientation choices do not make sense globally, they will cancel pairwise in our argument anyway.

\begin{eqnarray}
\gamma_1\ast \gamma_2\; &=& \quad
\xy
(0,0)*{
\begin{tikzpicture} 
\draw[very thick] (0,-.5) -- (0,.5);
\draw[draw =white, double=black, ultra thick , double distance=1.25pt] (-.5,0) -- (.5,0);
\end{tikzpicture}
};
\endxy
\quad = \quad
A\;
\xy
(0,0)*{
\begin{tikzpicture} 
\draw[very thick] (0,-.5) .. controls (0,0) .. (.5,0);
\draw[very thick] (-.5,0) .. controls (0,0) .. (0,.5);
\end{tikzpicture}
};
\endxy
\;+\;
A^{-1}\;
\xy
(0,0)*{
\begin{tikzpicture} 
\draw[very thick] (0,-.5) .. controls (0,0) .. (-.5,0);
\draw[very thick] (.5,0) .. controls (0,0) .. (0,.5);
\end{tikzpicture}
};
\endxy \nonumber \\
\psi(\gamma_1\ast \gamma_2)&=&
A\;
\xy
(0,0)*{
\begin{tikzpicture} 
\draw[very thick, ->] (0,-.5) .. controls (0,0) .. (.5,0);
\draw[very thick, ->] (-.5,0) .. controls (0,0) .. (0,.5);
\end{tikzpicture}
};
\endxy
\;+\;
A\;
\xy
(0,0)*{
\begin{tikzpicture} 
\draw[very thick, <-] (0,-.5) .. controls (0,0) .. (.5,0);
\draw[very thick, ->] (-.5,0) .. controls (0,0) .. (0,.5);
\end{tikzpicture}
};
\endxy
\;+\;
A\;
\xy
(0,0)*{
\begin{tikzpicture} 
\draw[very thick, <-] (0,-.5) .. controls (0,0) .. (.5,0);
\draw[very thick, <-] (-.5,0) .. controls (0,0) .. (0,.5);
\end{tikzpicture}
};
\endxy
\;+\;
A\;
\xy
(0,0)*{
\begin{tikzpicture} 
\draw[very thick, ->] (0,-.5) .. controls (0,0) .. (.5,0);
\draw[very thick, <-] (-.5,0) .. controls (0,0) .. (0,.5);
\end{tikzpicture}
};
\endxy
\nonumber \\ 
&+&
A^{-1}\;
\xy
(0,0)*{
\begin{tikzpicture} 
\draw[very thick,<-] (0,-.5) .. controls (0,0) .. (-.5,0);
\draw[very thick,<-] (.5,0) .. controls (0,0) .. (0,.5);
\end{tikzpicture}
};
\endxy
\;+\;
A^{-1}\;
\xy
(0,0)*{
\begin{tikzpicture} 
\draw[very thick,<-] (0,-.5) .. controls (0,0) .. (-.5,0);
\draw[very thick,->] (.5,0) .. controls (0,0) .. (0,.5);
\end{tikzpicture}
};
\endxy
\;+\;
A^{-1}\;
\xy
(0,0)*{
\begin{tikzpicture} 
\draw[very thick,->] (0,-.5) .. controls (0,0) .. (-.5,0);
\draw[very thick,->] (.5,0) .. controls (0,0) .. (0,.5);
\end{tikzpicture}
};
\endxy
\;+\;
A^{-1}\;
\xy
(0,0)*{
\begin{tikzpicture} 
\draw[very thick,->] (0,-.5) .. controls (0,0) .. (-.5,0);
\draw[very thick,<-] (.5,0) .. controls (0,0) .. (0,.5);
\end{tikzpicture}
};
\endxy
\nonumber \\
&=& AX_1 + AX_2 + AX_3 + AX_4 + A^{-1}X_5+ A^{-1}X_6+ A^{-1}X_7+ A^{-1}X_8 
 \nonumber \\
\psi(\gamma_1)\psi(\gamma_2)&=&
\quad
\xy
(0,0)*{
\begin{tikzpicture} 
\draw [very thick] (0,-.5) -- (0,-.1);
\draw[very thick,->] (0,.1) -- (0,.5);
\draw[very thick, ->] (-.5,0) -- (.5,0);
\end{tikzpicture}
};
\endxy
\;+\;
\xy
(0,0)*{
\begin{tikzpicture} 
\draw [very thick,<-] (0,-.5) -- (0,-.1);
\draw[very thick] (0,.1) -- (0,.5);
\draw[very thick, <-] (-.5,0) -- (.5,0);
\end{tikzpicture}
};
\endxy
\;+\;
\xy
(0,0)*{
\begin{tikzpicture} 
\draw [very thick,<-] (0,-.5) -- (0,-.1);
\draw[very thick] (0,.1) -- (0,.5);
\draw[very thick, ->] (-.5,0) -- (.5,0);
\end{tikzpicture}
};
\endxy
\;+\;
\xy
(0,0)*{
\begin{tikzpicture} 
\draw [very thick] (0,-.5) -- (0,-.1);
\draw[very thick,->] (0,.1) -- (0,.5);
\draw[very thick, <-] (-.5,0) -- (.5,0);
\end{tikzpicture}
};
\endxy
\nonumber \\
&=& \quad
A\;
\xy
(0,0)*{
\begin{tikzpicture} 
\draw[very thick, ->] (0,-.5) .. controls (0,0) .. (.5,0);
\draw[very thick, ->] (-.5,0) .. controls (0,0) .. (0,.5);
\end{tikzpicture}
};
\endxy
\;+\;
A\;
\xy
(0,0)*{
\begin{tikzpicture} 
\draw[very thick, <-] (0,-.5) .. controls (0,0) .. (.5,0);
\draw[very thick, <-] (-.5,0) .. controls (0,0) .. (0,.5);
\end{tikzpicture}
};
\endxy
\;+\;
A^{-1}\;
\xy
(0,0)*{
\begin{tikzpicture} 
\draw[very thick,<-] (0,-.5) .. controls (0,0) .. (-.5,0);
\draw[very thick,<-] (.5,0) .. controls (0,0) .. (0,.5);
\end{tikzpicture}
};
\endxy
\;+\;
A^{-1}\;
\xy
(0,0)*{
\begin{tikzpicture} 
\draw[very thick,->] (0,-.5) .. controls (0,0) .. (-.5,0);
\draw[very thick,->] (.5,0) .. controls (0,0) .. (0,.5);
\end{tikzpicture}
};
\endxy \nonumber \\
&=& AX_1 +AX_3+A^{-1}X_5+ A^{-1} X_7 \nonumber
\end{eqnarray}

Keeping the isolated crossing from above, fix a smoothing and orientation for every other crossing in $\gamma_1 \ast \gamma_2$. Then there will be 8 terms in  $\psi(\gamma_1 \ast \gamma_2)$ with these fixed choices: one corresponding to each of the eight smoothing and orientation choices $X_1, \ldots, X_8$ at the isolated crossing.  Notice that, for fixed choices elsewhere, the term $X_4$ exists (meaning local choices lead to a valid global orientation) if and only if the term $X_8$ exists. Further, notice that $AX_4 = -A^{-1} X_8$ by Relations \ref{relA_circ} and \ref{relA_Reid}. Equation \ref{eq:star} shows this calculation pictorially. 

\begin{eqnarray}  \label{eq:star}
AX_4 = A\;
\xy
(0,0)*{
\begin{tikzpicture} 
\draw[very thick, ->] (0,-.5) .. controls (0,0) .. (.5,0);
\draw[very thick, <-] (-.5,0) .. controls (0,0) .. (0,.5);
\end{tikzpicture}
};
\endxy
&=&
A\;
\xy
(0,0)*{
\begin{tikzpicture} 
\draw[very thick, ->] (0,-.5) .. controls (0,-.3) and (-.3,0) .. (-.5,0);
\draw[very thick, ->] (0,.5) .. controls (0,.3) and (.3,0)  .. (.5,0);
\draw [very thick,->] (.15,0) arc (0:360:.15); 
\end{tikzpicture}
};
\endxy
\quad = \quad
-A^{-1}\;
\xy
(0,0)*{
\begin{tikzpicture} 
\draw[very thick, ->] (0,-.5) .. controls (0,0) .. (-.5,0);
\draw[very thick, ->] (0,.5) .. controls (0,0)  .. (.5,0); 
\end{tikzpicture}
};
\endxy = -A^{-1}X_8
\end{eqnarray}

This means that $\psi(\gamma_1 \ast \gamma_2)$ is a sum over all possible choices away from the isolated crossing of terms of the form $A(X_1 + X_2 + X_3) + A^{-1}(X_5+ X_6+ X_7)$. Since $\psi(\gamma_1) \psi(\gamma_2)$ sums terms of the form $A(X_1 + X_3) + A^{-1}(X_5+ X_7)$ over the same set, we conclude that $\psi(\gamma_1 \ast\gamma_2) - \psi(\gamma_1) \psi(\gamma_2)$ reduces to a sum of terms of the form $AX_2 + A^{-1}X_6$ as shown in Equation \ref{reduct}. Our remaining work, then, is to prove that this sum is 0.

\begin{eqnarray}\label{reduct}
\psi(\gamma_1\ast \gamma_2)-\psi(\gamma_1)\psi(\gamma_2)&=&
A\;
\xy
(0,0)*{
\begin{tikzpicture} 
\draw[very thick, <-] (0,-.5) .. controls (0,0) .. (.5,0);
\draw[very thick, ->] (-.5,0) .. controls (0,0) .. (0,.5);
\end{tikzpicture}
};
\endxy
\;+\;
A^{-1}\;
\xy
(0,0)*{
\begin{tikzpicture} 
\draw[very thick,<-] (0,-.5) .. controls (0,0) .. (-.5,0);
\draw[very thick,->] (.5,0) .. controls (0,0) .. (0,.5);
\end{tikzpicture}
};
\endxy = AX_2 + A^{-1}X_6
\end{eqnarray}

Suppose there exists a term in $\psi(\gamma_1\ast \gamma_2)$ at which the isolating crossing looks like $X_2$. (The same argument will work in the case that the isolated crossing looks like $X_6$.) Again, the fact that this term exists means local smoothing and orientation choices elsewhere are compatible with $X_2$ at the isolated crossing. Follow the strand exiting the isolated crossing downwards. If the strand connects to a former crossing that looks like $X_3$ or $X_5$, recreate the crossing and continue following the strand downwards. This is shown in Figure \ref{recreate}. This shows that terms that look like $X_2$ at the isolated crossing cancel pairwise as do terms that look like $X_6$, and we conclude that $\psi(\gamma_1 \ast\gamma_2) - \psi(\gamma_1) \psi(\gamma_2) = 0$ as desired.

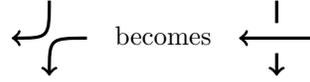
\begin{figure}[h] 
\[
\xy
(0,0)*{
\begin{tikzpicture} 
\draw[very thick, <-] (0,-.5) .. controls (0,0) .. (.5,0);
\draw[very thick, <-] (-.5,0) .. controls (0,0) .. (0,.5);
\end{tikzpicture}
};
\endxy
\quad
\text{becomes}
\quad
\xy
(0,0)*{
\begin{tikzpicture} 
\draw[very thick, ->] (.5,0) -- (-.5,0);
\draw [very thick, <-] (0, -.5) -- (0, -.2);
\draw[very thick] (0,.2) -- (0,.5);
\end{tikzpicture}
};
\endxy
\]
\caption{Recreating a crossing}\label{recreate}
\end{figure}

Continuing this process, one of two things happens. Either we eventually reach a former crossing that looks like $X_4$ or $X_8$ or we continue the crossing recovery process eventually leading back to the isolated crossing. In the latter case, the strand exiting the isolated crossing downwards must reconnect at the right or left since the term was globally oriented (and the recovery of crossings did not modify this orientation). This would mean that the isolated crossing originally came from a curve on $\tor$ intersecting itself. The curves $\gamma_1$ and $\gamma_2$ were skein elements, so self intersections are not possible.

This means that any term that looks like $X_2$ at the isolated crossing has another crossing that looks like $X_4$ or $X_8$. In fact, if $X_4$ at this other crossing leads to a valid global orientation, then so does $X_8$. This means that there is a pair of terms that are identical except that one looks like $X_4$ and one looks like $X_8$. Since the curves $\gamma_1$ and $\gamma_2$ are in generic position, the $X_4$ and $X_8$ terms will cancel by the calculation shown in Equation \ref{eq:star}.

\end{proof}

Example \ref{ex:ProofTechTh} illustrates the crossing recovery process described in the previous proof. In this example, the isolated crossing, which is marked with a solid circle, looks like $X_2$. Following the downward exiting strand, we encounter a crossing that looks like $X_5$, recover it, and then proceed. The next crossing, which is marked with a dotted circle,  looks like $X_8$.

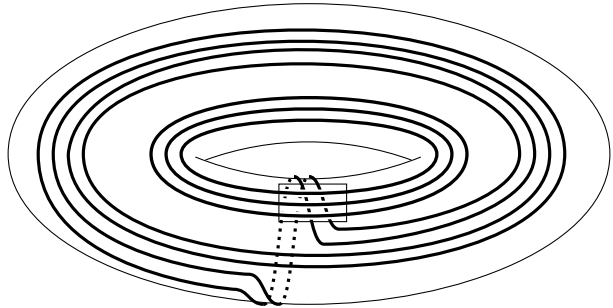
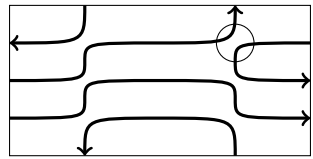
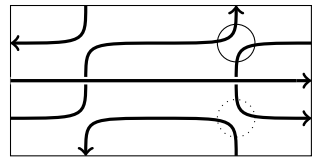
\begin{example} \label{ex:ProofTechTh}

\[
\xy (0,0)*{
  \begin{tikzpicture} [scale=1,  decoration={markings, 
                        mark=at position 0.6 with {\arrow{>}};    }]
\draw (-4,0) .. controls (-4,1.2) and (-2,2) .. (0,2) .. controls (2,2) and (4,1.2) .. (4,0) .. controls (4,-1.2) and (2,-2) .. (0,-2) .. controls (-2,-2) and (-4,-1.2) .. (-4,0);
\draw (-1.51,-0.04) arc (-112:-68:4);
\draw (1.36,-0.08) arc (70:110:4);
\draw [very thick] (0,-.3) .. controls (.2,-.3) and (.2,-1) .. (.4,-1) .. controls (1.4,-1) and (2.8,-.6) .. (2.8,0) .. controls (2.8,1.6) and (-3,1.6) .. (-3,0) .. controls (-3,-1.8)  and (3.2,-1.8) .. (3.2,0) .. controls (3.2,2) and (-3.4,2) .. (-3.4,0) .. controls (-3.4,-1.4) and (-1,-1.6) .. (-.8,-1.6) .. controls (-.6,-1.6) and (-.6,-2).. (-.4,-2);
\draw [dotted, very thick] (-.4,-2) .. controls (-.2,-2) and (-.2,-.3) .. (0,-.3);
\draw [very thick] (-.2,-.3) .. controls (0,-.3) and (0,-1.2) .. (.2,-1.2) .. controls (1.2,-1.2) and (3,-.8) .. (3,0) .. controls (3,1.8) and (-3,1.9) .. (-3.2,0) .. controls (-3.2,-2)  and (3.4,-2) .. (3.4,0) .. controls (3.4,2.2) and (-3.6,2.2) .. (-3.6,0) .. controls (-3.6,-1.6) and (-1,-1.8) .. (-1,-1.8) .. controls (-.8,-1.8) and (-.8,-2).. (-.6,-2);
\draw [dotted, very thick] (-.6,-2) .. controls (-.4,-2) and (-.4,-.3) .. (-.2,-.3);
\draw [draw =white, double=black, thick , double distance=1.25pt] (1.7,0) .. controls (1.7,.6) and (-1.7,.6) .. (-1.7,0) .. controls (-1.7,-.7) and (1.7,-.7) .. (1.7,0);  
\draw [draw =white, double=black, thick, double distance=1.25pt] (1.9,0) .. controls (1.9,.8) and (-1.9,.8) .. (-1.9,0) .. controls (-1.9,-.9) and (1.9,-.9) .. (1.9,0);  
\draw [draw =white, double=black, thick, double distance=1.25pt] (2.1,0) .. controls (2.1,1) and (-2.1,1) .. (-2.1,0) .. controls (-2.1,-1.1) and (2.1,-1.1) .. (2.1,0);  
\draw (-.4,-.4) rectangle (.5,-.9);
    \end{tikzpicture}};
\endxy
\]
gives for example:

\[
\xy (0,0)*{
  \begin{tikzpicture} [scale=1]
\draw (-2,-1) rectangle (2,1);
\draw [very thick, ->] (-1,1) .. controls (-1,.5) .. (-2,.5);
\draw [very thick, ->] (-2,0) .. controls (-1,0) .. (-1,.25) .. controls (-1,.5) .. (0,.5) .. controls (1,.5) .. (1,1);
\draw [very thick, ->] (-2,-.5) .. controls (-1,-.5) .. (-1,-.25) .. controls (-1,0) .. (0,0) .. controls (1,0) .. (1,-.25) .. controls (1,-.5) .. (2,-.5);
\draw [very thick, ->] (2,.5) .. controls (1,.5) .. (1,.25) .. controls (1,0) .. (2,0);
\draw [very thick, ->] (1,-1) .. controls (1,-.5) .. (0,-.5) .. controls (-1,-.5) .. (-1,-1);
\draw (1.25,.5) arc (0:360:.25);
    \end{tikzpicture}};
\endxy
\]
coming from:

\[
\xy (0,0)*{
  \begin{tikzpicture} [scale=1]
\draw (-2,-1) rectangle (2,1);
\draw [very thick, ->] (-1,1) .. controls (-1,.5) .. (-2,.5);
\draw [very thick, ->] (-2,-.5) .. controls (-1,-.5) .. (-1,0) .. controls (-1,.5) .. (0,.5) .. controls (1,.5) .. (1,1);
\draw [very thick, ->] (2,.5) .. controls (1,.5) .. (1,0) .. controls (1,-.5) .. (2,-.5);
\draw [very thick, ->] (1,-1) .. controls (1,-.5) .. (0,-.5) .. controls (-1,-.5) .. (-1,-1);
\draw [draw =white, double=black, thick, double distance=1.25pt] (-2,0) -- (1.8,0);
\draw [very thick, ->] (1.8,0) -- (2,0);
\draw (1.25,.5) arc (0:360:.25);
\draw[dotted] (1.25,-.5) arc (0:360:.25);
    \end{tikzpicture}};
\endxy
\]

\end{example}

\subsection{An alternative proof of the Frohman-Gelca formula}

With Theorem \ref{THM_iso} established, the Frohman-Gelca formula can now be proven via a straightforward calculation which highlights the importance of Chebyshev polynomials. This is accomplished by passing computations in $Sk(\tor)$ where crossings have multiple smoothings and thus products are complicated to $\Aa^{\Theta}$ where each crossing has a unique smoothing and thus products are much more tractable. 

Let $T_n$ be the $n^{th}$ Chebyshev polynomial of the first kind; it is well-known that $T_n(x+x^{-1})=x^{n}+x^{-n}$. This is a key property we will make use of in our proof of the Frohman-Gelca formula. Consider a basis element $(a,b)$ in $Sk(\tor)$ with $n=GCD(a,b)$. Recall that we have defined $(a,b)_T = T_n\left( \frac{a}{n}, \frac{b}{n} \right)$. With this information in mind, we offer the following oriented skein module proof of the Frohman-Gelca formula.

\begin{proof}[Proof of Theorem \ref{theorem::FG_Formula}]

Let $n=GCD(a,b)$ and let $m=GCD(c,d)$. The image under $\psi$ of $(a,b)_T \ast (c,d)_T$ can be computed as follows.
\begin{align*} 
\psi((a,b)_T \ast (c,d)_T) &= \psi((a,b)_T)\psi((c,d)_T) = \psi\left( T_n\left( \frac{a}{n}, \frac{b}{n} \right) \right)  \psi \left( T_m\left(  \frac{a}{m}, \frac{b}{m} \right)\right)\\
&= T_n\left( \gamma_{\frac{a}{n}, \frac{b}{n}} + \gamma^{-1}_{\frac{a}{n}, \frac{b}{n}}  \right)  T_m\left(  \gamma_{\frac{c}{m}, \frac{d}{m}} + \gamma^{-1}_{\frac{c}{m}, \frac{d}{m}}\right)\\
&=\left( \left(\gamma_{\frac{a}{n}, \frac{b}{n}}\right) ^n+ \left( \gamma_{\frac{a}{n}, \frac{b}{n}}\right)^{-n} \right) \left( \left(\gamma_{\frac{c}{m}, \frac{d}{m}}\right) ^m+ \left( \gamma_{\frac{c}{m}, \frac{d}{m}}\right)^{-m} \right) \\
&=(\gamma_{a,b} + \gamma_{-a,-b})(\gamma_{c,d} +\gamma_{-c,-d})\\
&= A^{\tiny{{ }^{\setlength{\arraycolsep}{1pt} \begin{vmatrix}
a & c \\
b & d
\end{vmatrix}}}} (\gamma_{a-c,b-d}+\gamma_{-a+c,-b+d}) 
 +A^{\tiny{{ }^{-\setlength{\arraycolsep}{1pt} \begin{vmatrix}
a & c \\
b & d
\end{vmatrix}}}} (\gamma_{a+c,b+d}+\gamma_{-a-c,-b-d})\\
&= A^{\tiny{{ }^{\setlength{\arraycolsep}{1pt} \begin{vmatrix}
a & c \\
b & d
\end{vmatrix}}}} \psi((a-c,b-d)_T)
 +A^{\tiny{{ }^{-\setlength{\arraycolsep}{1pt} \begin{vmatrix}
a & c \\
b & d
\end{vmatrix}}}} \psi((a+c,b+d)_T)\\ 
&= \psi \left(A^{\tiny{{ }^{\setlength{\arraycolsep}{1pt} \begin{vmatrix}
a & c \\
b & d
\end{vmatrix}}}} (a-c,b-d)_T
 +A^{\tiny{{ }^{-\setlength{\arraycolsep}{1pt} \begin{vmatrix}
a & c \\
b & d
\end{vmatrix}}}} (a+c,b+d)_T\right)\\
\end{align*}

The Frohman-Gelca formula now follows from the fact that $\psi$ is an isomorphism.

\end{proof}


\bibliographystyle{plain}
\bibliography{bibliothese}

%
\end{document}